\documentclass{article}

\usepackage{amsmath, amsthm, amssymb}
\usepackage[colorlinks, linkcolor=blue]{hyperref}
\usepackage[utf8]{inputenc}

\usepackage{todonotes}

\usepackage{xcolor}

\newtheorem{theorem}{Theorem}[section]
\newtheorem{lemma}[theorem]{Lemma}
\newtheorem{corollary}[theorem]{Corollary}
\newtheorem{proposition}[theorem]{Proposition}

\theoremstyle{definition}
\newtheorem{definition}[theorem]{Definition}
\newtheorem{example}[theorem]{Example}
\newtheorem{remark}[theorem]{Remark}

\newcommand{\Var}{\operatorname{Var}}
\newcommand{\ges}{\geqslant}
\newcommand{\les}{\leqslant}

\begin{document}

\title{The Lindeberg theorem for Gibbs-Markov dynamics}
\author{Manfred Denker, Samuel Senti, Xuan Zhang}
\date{}
\maketitle

\begin{abstract}
A dynamical array consists of a family of functions $\{f_{n,i}: 1\les i\les k_n, n\ges 1\}$ and a family of initial times $\{\tau_{n,i}: 1\les i\les k_n, n\ges 1\}$. For a dynamical system $(X,T)$ we identify  distributional limits for sums of the form
$$  S_n= \frac 1{s_n}\sum_{i=1}^{k_n} [f_{n,i}\circ T^{\tau_{n,i}}-a_{n,i}]\qquad n\ges 1$$
for suitable (non-random) constants $s_n>0$ and $a_{n,i}\in 
\mathbb R$. We derive a Lindeberg-type central limit theorem for dynamical arrays. Applications include new central limit theorems for functions which are not locally Lipschitz continuous and central limit theorems for statistical functions of time series obtained from Gibbs-Markov systems.
Our results, which hold for more general dynamics, are stated in the context of Gibbs-Markov dynamical systems for convenience.
\end{abstract}

\section{Introduction}
Probabilistic methods have been used for a long time in connection with number theory and some of these applications have formulations in terms of dynamical systems, but it took more than 50 years to realize the general importance for dynamics.
Continued fraction is a typical example of such a common approach in number theory and dynamics. While the ergodic theorem  has a direct counterpart in Kolmogorov's strong law of large numbers, the classical de Moivre-Laplace central limit theorem describing its fluctuation about the mean has none.
Today, central limit theorems (CLT)  are widespread in the study of fluctuations of Birkhoff ergodic sums in dynamical systems. Ideas borrowed from probability theory such as stationary mixing processes \cite{IbragimovLinnik1971} and Gordin's martingale-coboundary method \cite{Gordin1969} are commonly used to prove central limit theorems.
The survey paper \cite{Denker1989} has a comprehensive list of references up to 1986. More recently, Chazottes \cite{Chazottes2015} reviewed probabilistic laws for ergodic sums in dynamical systems modeled by Young towers (\cite{Young1998}). A central limit theorem for Markov fibred systems with the Schweiger property was proven in \cite{AaronsonDenkerUrbanski1993}. Examples of such systems include parabolic rational maps, Young towers and Gibbs-Markov maps. A CLT for general rational maps was proven in \cite{DenkerPrzytyckiUrbanski1996} using Gordin's method. All these results are concerned with Birkhoff sums. In order to obtain a CLT, the observables in the  Birkhoff sums are usually assumed to be H\"older continuous. This is partly due to the popular spectral gap method, which usually holds on Banach spaces endowed with H\"older norms. On the other hand it is still an open problem to determine the class of functions in $L^2$ satisfying the CLT. 

From an applied viewpoint, ergodic sums provide only a limited method to draw conclusions on a dynamical system. A much wider approach is formulated  in terms of  design of experiments where different time series and their interplay are considered. This leads to the need of analyzing arrays containing different ergodic sums. To our knowledge this concept was first formulated for maps of the interval and some special statistical functionals in \cite{Denker1982}.  Recently, \cite{HaydnNicolVaientiEtAl2013} used a special form of such an approach to obtain CLT for  shrinking targets. In other directions, one should also note CLTs in the settings of random dynamics or sequential dynamics such as in \cite{ConzeRaugi2007, CohenConze2017}. Lindeberg's central limit theorem deals with arrays of independent random variables, i.e. families of random variables defined on row-wise different probability spaces. We formulate Lindeberg's central limit theorem for dynamical arrays, and prove CLTs for arrays in dynamical systems, here Gibbs-Markov maps.
Examples include certain countable state Markov chains and Markov maps of the unit interval given in \cite{AaronsonDenker2001} as well as parabolic rational maps in \cite{AaronsonDenkerUrbanski1993}. 
Other examples can be found in \cite{AaronsonDenker1999}. 
We use two classical methods: the characteristic operators approach as in \cite{Rousseau-Egele1983} and Lindeberg's method as in \cite{Lindeberg1922} for blocks to prove CLTs. It will be clear from the proof that our results can be extended to more general systems since only spectral properties of transfer operators and metric properties are taken from Gibbs-Markov systems. It is also clear how to extend the results to Young towers over Gibbs-Markov maps. Recent development can be found in \cite{Thomine2014}. For simplicity, we keep our discussions restricted to Gibbs-Markov systems.

Dynamical arrays have many applications. For instance, one may use an array of H\"older continuous observables to approximate Birkhoff sums of observables of lesser regularity. An example is given in Corollary \ref{cor:lesreg}. In comparison, note that Gou\"ezel \cite{Goueezel2010} proved a CLT for Birkhoff sums of observables with H\"older norm in $L^\eta$, where $0<\eta<1$. In another paper \cite{DenkerSentiZhang2017} we have used CLTs for arrays to study fluctuations for ergodic sums over periodic orbits. Another possible application is through coupling Birkhoff sums of different dynamical systems. 
 
We recall some background material on Gibbs-Markov systems and spectral properties of their transfer operators in Sections \ref{sec:GibbsMarkov} and \ref{sec:transfer}. Section \ref{sec:clt} contains a CLT (Theorem \ref{thm:CLTbd}) for sequences of Birkhoff sums  $\sum_{i=1}^{k_n} f_{n}\circ T^i$ ($n\ges 1$). Although this is a special dynamical array the central limit theorem is treated separately since the method of proof is different from the other main theorem and may have future applications to other dynamical systems. Such theorems provide central limit theorems for Birkhoff sums $\sum_{i=1}^n f\circ T^i$ for certain functions which are not Lipschitz (or H\"older) continuous. We provide one easy example and others are not difficult to obtain. Section \ref{sec:lindeberg} contains the Lindeberg CLT (Theorem \ref{thm:CLTdarray}) for dynamical arrays. Here we deal with the CLT in its most general form as loosely formulated in the abstract. The precise statement and assumptions are presented at the beginning of this section. There are many applications of this theorem. In Section \ref{sec:stat} we provide one of them by showing the asymptotic normality of the Wilcoxon two sample rank statistics. Other examples are given in \cite{Zhang2015} and will be derived elsewhere. To get an idea of the scope of other possible applications, one may consult \cite{Ferguson1996} or similar expositions.   
 
 \section{Gibbs-Markov systems}
 \label{sec:GibbsMarkov}
 Gibbs-Markov systems were first formulated in \cite{AaronsonDenker2001}. Let $(\Omega,\mathcal B,\mu,T)$ denote a nonsingular transformation of a standard probability
space. Consider a countable partition $\alpha$ of $\Omega\mod \mu$, $\alpha=\{a_i:i\in I\}$, and denote the $\sigma$-algebra generated by $\alpha$ by $\sigma(\alpha)$. For $x,y\in \Omega$, define 
$$s(x,y):=\min_{n\in\mathbb N}\{n+1:T^n(x) \text{ and } T^n(y) \text{ belong to different elements of } \alpha \}.$$
For any $r\in(0,1)$, set $r(x,y):= r^{s(x,y)}$ on $\Omega$, which will become a metric. We use the same letter $r$ to express the dependence of the metric on the choice of $r$. It will be clear in the subsequent context when $r$ represents a number or a metric.
\begin{definition}\label{def:gibbsmarkov}A quintuple $(\Omega,\mathcal{B},\mu,T,\alpha)$ is called a Gibbs-Markov map (or system) if the following four conditions hold modulo $\mu$. 
\begin{enumerate}
	\item $\alpha$ is a strong generator of $\mathcal{B}$ by $T$, i.e. $\sigma(\{T^{-n}\alpha:n\geqslant 0\})=\mathcal{B}$.
	\item For every $a\in\alpha$, $Ta\in \sigma(\alpha)$ and the restriction $T|_{a}$ is invertible and (two-sided) nonsingular. 
	\item $\inf\limits_{a\in\alpha}\mu(Ta)>0$.
	\item  For each $n\geqslant 1$ and $a\in\bigvee_{i=0}^{n-1}T^{-i}\alpha$, denote the $\mu$-nonsingular inverse branch $T^{-n}|_{T^n a}$ by $v_a: T^n a\rightarrow a$ and its Radon-Nikodym derivative by $v'_a$. Then, there exist $r\in (0,1)$ and $M>0$ such that for any $n\geqslant1, a\in\bigvee_{i=0}^{n-1}T^{-i}\alpha$ and $x,y\in T^na$,
			\begin{equation}\label{eq:distort}\left|\dfrac{v'_a(x)}{v'_a(y)}-1\right|\leqslant M\cdot r(x,y).
			\end{equation}
\end{enumerate}
\end{definition}

\begin{remark}
Usually we do not specify $\mathcal B$ and write only $(\Omega,\mu, T, \alpha)$. Also note that a number $r$ and hence the metric $r(\cdot,\cdot)$ are determined within the definition of a Gibbs-Markov map.
\end{remark}
We will work with the following Banach spaces: given a Gibbs-Markov system $(\Omega,\mu,T,\alpha)$ and any partition $\rho$ of $\Omega$, the H\"older norm of a function $f:\Omega\to \mathbb C$ subject to the partition $\rho$ is defined by
\begin{equation*}
D_\rho f:=\sup\limits_{b\in\rho}\sup\limits_{x, y\in b, x\neq y}\frac{|f(x)-f(y)|}{r(x,y)},
\end{equation*}
where  $\sup$ is understood to be taken $\mu$ almost everywhere. 
Denote the usual $L^q$-norm by $\|\cdot\|_q$, $1\leqslant q\leqslant\infty$. Then
\begin{equation*}
\|f\|_{\infty,\rho}:=\|f\|_\infty+D_{\rho} f
\end{equation*}
defines a larger norm. Denote the subspace consisting of functions of finite $\|\cdot\|_{\infty,\rho}$ norm by $L^\infty_\rho$. It is standard to show that $L^\infty_\rho$ is a Banach space. 

\begin{remark}
Throughout this paper we will always assume that $(\Omega,\mu, T, \alpha)$ is a topologically mixing and measure-preserving Gibbs-Markov system, here topologically mixing means that
for any $a, b\in\alpha$, there is $n_{a,b}\in\mathbb N$ such that for every $n>n_{a,b}$, $b\subset T^n a$.
\end{remark}

 \section{Transfer operator and characteristic function operator}\label{sec:transfer}
We continue setting up the theory for a Gibbs-Markov system $(\Omega,\mu, T, \alpha)$ by introducing its transfer operators. Since $T\alpha\subset \sigma(\alpha)$, it follows that for every $n\in\mathbb N$, $T^n(\bigvee_{i=0}^{n-1}T^{-i}\alpha)=T\alpha$. Fix a partition $\beta$ (which may be coarser than $\alpha$) such that 
\begin{equation}\label{eq:partbeta}\sigma(T\alpha)=\sigma(\beta).\end{equation}
 The Perron-Frobenius transfer operator $\mathcal{L}_T:L^1(\mu)\rightarrow L^1(\mu)$ is defined by:
 $$  \mathcal{L}_Tf:=\sum_{b\in\beta}{\bf 1}_b\sum_{a\in\alpha,Ta\supset b}v'_a\cdot f\circ v_a,$$
and  the transfer operator for $T^n$ is hence of the form $$\mathcal{L}_{T^n}f=\sum_{b\in\beta}{\bf 1}_b\sum_{a\in\bigvee_{i=0}^{n-1}T^{-i}\alpha,T^na\supset b}v'_a\cdot f\circ v_a.$$
$\mathcal L_T$ satisfies and is uniquely characterized by:
  \begin{equation*}\label{eq:transfer}
   \int_\Omega \mathcal{L}_Tf\cdot gd\mu=\int_\Omega f\cdot g\circ Td\mu, \qquad \forall g\in L^\infty(\mu).
  \end{equation*}
It can be easily derived from the above equation that $$\mathcal{L}_{T^n}=\mathcal{L}^{n}_T$$ and since $\mu$ is $T$-invariant 
\begin{equation}\label{eq:ev1}\mathcal{L}_T {\mathbf 1}={\bf 1}.
\end{equation} 
We will use $\mathcal{L}$ for $\mathcal{L}_T$ when $T$ is fixed.  
 
 Given a measurable function $f:\Omega\to\mathbb R$ and $t\in\mathbb R$, the characteristic function operator $\mathcal{L}_{f,t}:L^1(\mu)\to L^1(\mu)$ is defined as:
 $$\mathcal{L}_{f,t}g:=\mathcal{L}(e^{itf}\cdot g), \quad \forall g\in L^1(\mu).$$
 Then $$ \mathcal L_{f,0}=\mathcal L.$$
Note that $D_\alpha(\cdot) \les D_\beta(\cdot) \les D_\Omega(\cdot) \les\max\{\frac{2\|\cdot\|_\infty}r, D_\alpha(\cdot) \}$, hence the functions in $L_\beta^\infty$ have finite $D_\alpha$ norms, and the norm $\|\cdot\|_{\infty,\beta}$ is equivalent to the norm $\|\cdot\|_{\infty,\alpha}$ and to the norm $\|\cdot\|_{\infty,\Omega}$. From now on, we write for simplicity
$$L:=L_\beta^\infty, \qquad \|\cdot\|:=\|\cdot\|_{\infty,\beta}.$$ As no confusion should appear, we use the same notation $\|\cdot\|$ for the operator norm on $L$. It is not hard to see that $\mathcal L$ and $\mathcal L_{f,t}$ are both bounded linear operators on $L$. In fact, we have the following estimates.
 \begin{lemma}[{\cite[Proposition 2.1, Theorem 2.4]{AaronsonDenker2001}}]\label{lem:df} There exist constants $M$ and $M_1$ such that for any $f, g\in L$ and $s,t\in \mathbb R$, we have
  \begin{equation}\label{eq:df}
   \|\mathcal L^n_{f,t}g\|\leqslant (M+M_1D_\alpha e^{itf})(r^nD_\beta g+\|g\|_1)
  \end{equation}
  and
  \begin{equation*}
   \|\mathcal L_{f,t}-\mathcal L_{f,s}\|\leqslant  (M+M_1 |s|D_\alpha f)(\|e^{i(t-s)f}-{\mathbf 1}\|_1+|t-s|D_\alpha f).
  \end{equation*}
  In particular, let $s=0$, we have for every $t\in\mathbb R$ and $f\in L$,
  \begin{equation*}\label{eq:close}
   \|\mathcal L_{f,t}-\mathcal L\|\leqslant 2M|t|\cdot\|f\|.
  \end{equation*}
 \end{lemma}

The following lemma, adapted from \cite[Proposition 3]{Rousseau-Egele1983}, provides the Taylor expansion of $\mathcal L_{f,t}$ around $\mathcal L$.
 \begin{lemma}
 For any $f\in L$ and $t\in\mathbb R$, there exist bounded linear operators $\mathcal L_f^{(n)}$ and $\mathcal L_{f,t}^{(n)}$ on $L$ such that
 \begin{equation*}
  \mathcal L_{f,t}=\mathcal L +\sum_{n=1}^\infty \frac{t^n}{n!}\mathcal L_f^{(n)}
 \end{equation*}
 converges absolutely with $\|\mathcal L_f^{(n)}\|\leqslant \|\mathcal L\|\|f\|^n$, and for every $m\in\mathbb N$,
 \begin{equation}\label{eq:expandL}
  \mathcal L_{f,t}=\mathcal L +t\mathcal L_f^{(1)}+\cdots +\frac{t^{m-1}}{(m-1)!}\mathcal L_f^{(m-1)}+\mathcal L_{f,t}^{(m)}
 \end{equation}
 with $\|\mathcal L_{f,t}^{(m)}\|\leqslant\|\mathcal L\||t|^m\|f\|^m e^{|t|\|f\|}.$
 \end{lemma}
\begin{remark}
 $\mathcal L_f^{(n)}$ are just derivatives of $\mathcal L_{f,t}$ around $t=0$ when $f$ is fixed.
\end{remark}
 \begin{proof} Since for any $f,g\in L$,
  $$\|f\cdot g\|_\infty\leqslant \|f\|_\infty \cdot \|g\|_\infty \text{ and }$$
  $$D_\beta (f\cdot g) \leqslant  D_\beta(f) \cdot\|g\|_\infty+\|f\|_\infty\cdot D_\beta(g),$$
  we have \begin{equation}\label{eq:normprod}\|f\cdot g\|\leqslant\|f\|\cdot\|g\|.\end{equation}
  Therefore $\|f^n\|\leqslant \|f\|^n$, and
  $$\|\sum_{n=0}^\infty\frac{(it)^n}{n!}\mathcal L (f^ng)\|\leqslant \sum_{n=0}^\infty\frac{1}{n!}\|\mathcal L\|\|f\|^n|t|^n\|g\|$$
  converges absolutely. Let $\mathcal L_f^{(n)}(\cdot):=i^n\mathcal L(f^n\cdot)$ and the expansion follows.
 \end{proof}

One of the underlying tools throughout this paper is the spectral gap property of the transfer operator $\mathcal L$, which has been proved in \cite[Theorem 1.6]{AaronsonDenker2001}.  We explain it briefly here. Because that $L$-bounded sets are precompact in $L^1$ and because of Lemma \ref{lem:df} and a theorem of Ionescu-Tulcea and Marinescu (\cite{IonescuTulceaMarinescu1950}),  $\mathcal L$ is quasi-compact on $L$. Notice that 1 is an eigenvalue of $\mathcal L$ and is a maximal eigenvalue of $\mathcal L$ on $L$ by \eqref{eq:ev1} since $\mathcal L$ contracts $L^\infty$. Also it is known that a (topologically) mixing Gibbs-Markov map is exact (\cite{AaronsonDenkerUrbanski1993, AaronsonDenker2001}) and hence is strong-mixing when it is measure-preserving. So $1$ is the unique maximal eigenvalue and is simple. Hence the transfer operator $\mathcal L$ of a mixing Gibbs-Markov map has the spectral gap property on $L$, namely one can decompose $\mathcal L$ on $L$ as
 \begin{equation}\label{eq:trsfdecomp}
 \mathcal{L}=P+N
 \end{equation}
 so that $Pf=\int_\Omega fd\mu$, $P N=N P=0$ and $\mathfrak r(N)<1$, where $\mathfrak r(\cdot)$ denotes the spectral radius. $P$ is the eigenprojection of $\mathcal L$ with respect to the eigenvalue $1$ and the spectrum of $N$ is all the remaining spectrum of $\mathcal L$. For any complex number $z$ not in the spectrum of $\mathcal L$, denote by $R(z;\mathcal L)$ the resolvent of $\mathcal L$, $(z\mathcal I-\mathcal L)^{-1}$. According to \cite[VII]{DunfordSchwartz1958}, one can calculate $P$ and $N$ by integrating the product of the resolvent and suitable analytic functions on neighborhoods of the spectrum of $\mathcal L$. In fact,  
\begin{equation}\label{eq:PNresolv}
P=\frac1 {2\pi i}\int_{C_1}R(z;\mathcal L)dz, \quad N=\frac1 {2\pi i}\int_{C_2}z R(z;\mathcal L)dz
\end{equation} where $C_1$ is a small circle around $1$ of radius  $\frac{1-\frak r(N)}3$ and $C_2$ is a circle around $0$ of radius $\frac{1+2\frak r(N)}{3}$ so that $C_1$ and $C_2$ are disjoint and that the spectrum of $\mathcal L$ except for the eigenvalue $1$ is totally contained within $C_2$. For every positive integer $k$,
$$N^k= \frac1 {2\pi i}\int_{C_2}z^k R(z;\mathcal L)dz.$$
We will call $$\rho_1:=\frac{1-\frak r(N)}3,\quad \rho_2:=\frac{1+2\frak r(N)}{3}.$$
By perturbation theory, the characteristic function operators also satisfy the above properties.
 \begin{lemma}\label{lem:decomp}
  There exists a real number $a>0$ such that if  $ |t|\cdot\|f\|<a $
  then $\mathcal L_{f,t}$ has the spectral gap property on $L$ with the decomposition:
  $$\mathcal L_{f,t}=\lambda_{f,t}P_{f,t}+N_{f,t}$$
  where
  \begin{enumerate}
   \item$\lambda_{f,t}$ is the unique eigenvalue of the largest modulus of $\mathcal L_{f,t}$, $\lambda_{f,t}$ is a simple eigenvalue and $|\lambda_{f,t}|\in(1-\rho_1,1+\rho_1)$;
   \item $P_{f,t}$ is the eigenprojection of $\mathcal L_{f,t}$ with respect to $\lambda_{f,t}$, in the form
$$P_{f,t}=\frac1 {2\pi i}\int_{C_1}R(z;\mathcal L_{f,t})dz;$$
  \item $\frak r(N_{f,t})<\rho_2$ and $P_{f,t}N_{f,t}=N_{f,t}P_{f,t}=0$, in fact, $$N_{f,t}=\frac1 {2\pi i}\int_{C_2}z R(z;\mathcal L_{f,t})dz;$$
   \item fix an $f\in L$, then $t\mapsto \lambda_{f,t}, t\mapsto P_{f,t}$ and $t\mapsto N_{f,t}$ are analytic on $(-a/\|f\|, a/\|f\|)$.
  \end{enumerate}
Here $C_1, C_2$ are the same circles as in \eqref{eq:PNresolv}.
 \end{lemma}
 
  This lemma is essentially \cite[Proposition 4]{Rousseau-Egele1983}. For our purpose we take expansions of the operators to higher orders in the next lemma.
  
 \begin{lemma}\label{lem:expansion}
  There exist constants $M>0$ and $0<a<M^{-1}$ such that if $|t|\cdot \|f\|<a$ is small enough, then
  \begin{enumerate}
   \item $P_{f,t}$ has an expansion $$ P_{f,t}=P+t P_f^{(1)}+\frac{t^2}{2} P_f^{(2)}+P_{f,t}^{(3)},$$ where $\|P^{(i)}_f\|\leqslant M\|f\|^i$, for $i=1,2$, and $\|P_{f,t}^{(3)}\|\leqslant M|t|^3\|f\|^3 e^{3|t|\|f\|}$;
   \item similarly, $\lambda_{f,t}$ expands as $$\lambda_{f,t}=1+t \lambda_f^{(1)}+\frac{t^2}{2} \lambda_f^{(2)}+\lambda_{f,t}^{(3)},$$ where $|\lambda^{(i)}_f|\leqslant M\|f\|^i$, for $i=1,2$, and $|\lambda_{f,t}^{(3)}|\leqslant M|t|^3\|f\|^3 e^{3|t|\|f\|}$;
   \item for all $n\in \mathbb N$, $$\|N^n_{f,t}{\bf 1}\|\leqslant \rho_2^n \frac{M|t|\|f\|}{1-M|t|\|f\|}.$$
  \end{enumerate}  
 \end{lemma}

 \begin{proof} We use notations from Lemma \ref{lem:decomp}. Notice that for any $z$ in the resolvent set of $\mathcal L$, if $|t|\cdot \|f\|$ is so small that  $\|\mathcal L_{f,t}-\mathcal L\|\cdot\|R(z, \mathcal L)\|<1$ then $z$ is also in the resolvent set of $\mathcal L_{f,t}$ and
  \begin{equation}\label{eq:resolvent}
   R(z;\mathcal L_{f,t})=R(z;\mathcal L)\sum_{n=0}^\infty((\mathcal L_{f,t}-\mathcal L)R(z;\mathcal L))^n
  \end{equation}
  converges absolutely.
  \begin{enumerate}
  \item We use the resolvent equation \eqref{eq:resolvent} to calculate $P_{f,t}$ as follows. Choose $a$ small enough such that $\|\mathcal L_{f,t}-\mathcal L\|\cdot \sup_{z\in C_1}\|R(z,\mathcal L)\|<1$ whenever $|t|\|f\|<a$. Then, denoting by $R:=R(z;\mathcal L)$,
  \begin{align*}
   P_{f,t}&=\frac{1}{2\pi i}\int_{C_1} R(z;\mathcal L_{f,t})dz\\
 &\overset{\eqref{eq:resolvent}}{=}\frac{1}{2\pi i}\int_{C_1}R\sum_{n=0}^\infty((\mathcal L_{f,t}-\mathcal L)R)^n dz\\
   &=P+\frac{1}{2\pi i}\int_{C_1}R\sum_{n=1}^\infty((\mathcal L_{f,t}-\mathcal L)R)^n dz\\
   &\overset{\eqref{eq:expandL}}{=}P+t\frac{1}{2\pi i}\int_{C_1}R\mathcal L_f^{(1)} Rdz
+ \frac{t^2}{2}\frac{1}{2\pi i}\int_{C_1}\left\{R\mathcal L_{f}^{(2)} R +2R\left(\mathcal L_f^{(1)}R\right)^2\right\}dz\\
&\quad+ \frac{1}{2\pi i}\int_{C_1}\left\{R\mathcal L_{f,t}^{(3)}R+ tR\mathcal L_{f,t}^{(2)}R\mathcal L_f^{(1)}R +tR\mathcal L_f^{(1)}R\mathcal L_{f,t}^{(2)}R +R\left(\mathcal L_{f,t}^{(2)}R\right)^2 \right.\\
&\quad \left.+ R\sum_{n=3}^\infty\left(\mathcal L_{f,t}^{(1)}R\right)^n\right\} dz.
\end{align*}
   Define corresponding operators to write the last equation in the form $$P_{f,t}=P+t P_f^{(1)}+ \frac{t^2}{2} P_f^{(2)}+ P_{f,t}^{(3)}.$$  
  Then there exist constants $M_1, M_2$ and $M_3$ such that when $|t|\|f\|<a$,
  \begin{align*}
   \|P_f^{(1)}\|&\leqslant M_3\|\mathcal L_f^{(1)}\|\leqslant M_1\|f\|,\\
   \|P_f^{(2)}\|&\leqslant M_3(\|\mathcal L_{f}^{(2)}\|+\|\mathcal L_f^{(1)}\|^2)\leqslant M_1\|f\|^2,\\
   \|P_{f,t}^{(3)}\|&\leqslant M_3(\|\mathcal L_{f,t}^{(3)}\|+|t|\cdot \|\mathcal L_{f,t}^{(2)}\|\cdot \|\mathcal L_f^{(1)}\|+ \|\mathcal L_{f,t}^{(2)}\|^2+\sum_{n=3}^\infty  M_2^n\|\mathcal L_{f,t}^{(1)}\|^n)\\
   &\leqslant M_1|t|^3\|f\|^3e^{3|t|\|f\|}.
  \end{align*} 
  \item Let $B$ be the Banach space of all bounded linear operators from $L$ to itself. Take a linear functional $\varphi\in B^*$ such that $\|\varphi\|_{B^*}=1$ and $\varphi(P)=1$. Then because
  \begin{equation*}\label{eq:eigenvalue}
   \mathcal L_{f,t} P_{f,t}=\lambda_{f,t} P_{f,t},
  \end{equation*}
  we have $$\lambda_{f,t}=\frac{\varphi(\mathcal L_{f,t} P_{f,t})}{\varphi(P_{f,t})}.$$
  Define $P_{f,t}^{(1)}= t P_f^{(1)}+ \frac{t^2}{2} P_f^{(2)}+ P_{f,t}^{(3)}$, then $\|P_{f,t}^{(1)}\|\leqslant M_1|t|\|f\|e^{|t|\|f\|}$. Choose $a$ small such that $ae^a<M_1^{-1}$ then when $|t|\|f\|<a$, $$\frac{1}{\varphi(P_{f,t})}=\frac{1}{1+\varphi(P_{f,t}^{(1)})}=\sum_{n=0}^{\infty}(-1)^n\varphi(P_{f,t}^{(1)})^n$$ since $\varphi(P_{f,t}^{(1)})\leqslant M_1|t|\|f\|e^{|t|\|f\|}$.
  Hence  the expansions of $\mathcal L_{f,t}$ and $P_{f,t}$ lead to the expansion of $\lambda_{f,t}$.
  \item The same resolvent equation is used in the calculation of $N_{f,t}$. Choose $a$ small enough such that $\|\mathcal L_{f,t}-\mathcal L\|\cdot \sup_{z\in C_2}(\|R(z,\mathcal L)\|+\|R(z,\mathcal L)\|^2)<1$ whenever $|t|\|f\|<a$, then
   \begin{align*}
    N^n_{f,t}&=\frac{1}{2\pi i}\int_{C_2}z^n R(z;\mathcal L_{f,t})dz\\
    &\overset{\eqref{eq:resolvent}}{=}\frac{1}{2\pi i}\int_{C_2} z^n R(z;\mathcal L)\sum_{m=0}^\infty((\mathcal L_{f,t}-\mathcal L)R(z;\mathcal L))^m dz\\
    &=N+ \frac{1}{2\pi i}\int_{C_2} z^n R(z;\mathcal L)\sum_{m=1}^\infty((\mathcal L_{f,t}-\mathcal L)R(z;\mathcal L))^m dz.
   \end{align*}
   Notice that $N{\bf 1}=0$, whence we have:
   \begin{align*}
    \|N^n_{f,t}{\bf 1}\|&\leqslant \frac{1}{2\pi}\|\int_{C_2} z^n R(z;\mathcal L)\sum_{m=1}^\infty((\mathcal L_{f,t}-\mathcal L)R(z;\mathcal L))^m dz\|\\
    &\leqslant \rho_2^n\frac{M_2|t|\|f\|}{1-M_2|t|\|f\|}
   \end{align*}
   for some constant $M_2$ with $a<M_2^{-1}$.
  \end{enumerate} 
 \end{proof}
  
When $f$ is fixed, since $\lambda_{f,t}$ is analytic with respect to $t$ around $0$ (Lemma \ref{lem:decomp}), the coefficients $\lambda_f^{(i)}$ in the expansion of $\lambda_{f,t}$ are just the corresponding derivatives of $\lambda_{f,t}$ with respect to $t$ at $t=0$. They can be calculated in the following manner. 
 \begin{lemma}[{\cite[Lemmes 2, 3, 6]{Rousseau-Egele1983}}]\label{lem:evexpansion}
  Let $f\in L$ with $\int_\Omega fd\mu=0$. Then 
  \begin{equation*}
   \lambda_f^{(1)}=0, \quad \lambda_f^{(2)}=-\lim_{n\to\infty}\frac{1}{n}\int_\Omega (\sum_{m=0}^{n-1} f\circ T^m)^2d\mu.
  \end{equation*}
  The limit exists, and $\lambda_f^{(2)}\neq0$ if and only if $f$ is not of the form $\varphi\circ T-\varphi$ for any $\varphi\in L$.
 \end{lemma}
 \begin{proof}
 One only needs to notice that the transfer operator and its spectral decomposition in \cite{Rousseau-Egele1983} share the same spectral properties and expand in the same way as in our settings when $f$ is fixed, to which the proofs of \cite[Lemmes 2, 3, 6]{Rousseau-Egele1983} refer, hence the proofs carry over to our settings.
 \end{proof}
  The  asymptotic variance of $f$ will be denoted by $$\sigma_f^2:=\lim\limits_{n\to\infty}\frac{1}{n}\int_\Omega(\sum\limits_{m=0}^{n-1} f\circ T^m)^2d\mu.$$
   
 \section{A CLT for arrays of Birkhoff sums}
 \label{sec:clt}
As mentioned in the introduction, in this section we prove a central limit theorem based on the tools developed in Section \ref{sec:transfer} for arrays of functions in Gibbs-Markov systems. The following theorem appeared as part of the third author's PhD thesis \cite{Zhang2015}.  \begin{theorem}\label{thm:CLTbd}
  Consider a Gibbs-Markov system $(\Omega, \mu, T, \alpha)$, a sequence $\{f_n\}\subset L$ with $\int_\Omega f_n d\mu=0$ and not of the form $\varphi\circ T-\varphi$ for any $\varphi\in L$ and a sequence of positive integers $k_n\to \infty$. If $$\lim_{n\to\infty}\frac{\|f_n\|^3}{\sqrt{k_n}\sigma_n^3}=0,$$
  where $\sigma_n^2:=\sigma_{f_n}^2$ is the asymptotic variance of $f_n$, then $$ \frac{f_n+f_n\circ T+...+f_n\circ T^{k_n-1} }{\sqrt{k_n\sigma_n^2}}$$ converges in distribution to the standard normal law $\mathcal N(0,1)$. 
  \end{theorem}
 
 \begin{proof}
   Let $S_n=f_n+f_n\circ T+...+f_n\circ T^{k_n-1}$. It can be easily verified that
   $$\mathcal L^{k_n}_{f_n,t}{\bf 1}=\mathcal L^{k_n} e^{itS_n}.$$
   Let $t_n=\frac{t}{\sqrt{k_n}\sigma_n}$, then we have
   \begin{align*}
    \int_\Omega e^{it\frac{S_n}{\sqrt{k_n}\sigma_n}}d\mu&=\int_\Omega e^{it_nS_n}d\mu=\int_\Omega \mathcal L^{k_n} e^{it_nS_n} d\mu\\
    &=\int_\Omega \mathcal L^{k_n}_{f_n,t_n}{\bf 1} d\mu.
   \end{align*}
   Since by assumption as $n\to\infty$, $$ |t_n|\|f_n\|=|t|\frac{\|f_n\|}{\sqrt{k_n}\sigma_n}\to 0,$$
   when $n$ is large, according to Lemma \ref{lem:decomp} 
   $$ \int_\Omega e^{it\frac{S_n}{\sqrt{k_n}\sigma_n}}d\mu=\lambda^{k_n}_{f_n,t_n}\int_\Omega P_{f_n,t_n}{\bf 1}d\mu+\int_\Omega N^{k_n}_{f_n,t_n}{\bf 1}d\mu.$$
   Lemma \ref{lem:expansion} implies that
   $$ \lim_{n\to\infty}\int_\Omega P_{f_n,t_n}{\bf 1}d\mu=\int_\Omega P{\bf 1}d\mu=1$$
   and for some constant $M$, $$\left|\int N^{k_n}_{f_n,t_n}{\bf 1}d\mu\right|\leqslant \rho_2^{k_n}\frac{M|t_n|\|f_n\|}{1-M|t_n|\|f_n\|}\to 0.$$
   By Lemma \ref{lem:expansion} and Lemma \ref{lem:evexpansion}, when $n$ is large,
   \begin{align*}
    \lambda_{f_n,t_n}&=1-\frac{1}{2}\sigma_{n}^2t_n^2+ \lambda_{f_n,t_n}^{(3)}\\
    &=1-\frac{t^2}{2k_n}+ \lambda_{f_n,t_n}^{(3)},
   \end{align*}
where $|\lambda_{f_n,t_n}^{(3)}|\leqslant M|t_n|^3\|f_n\|^3 e^{3|t_n|\|f_n\|}$.
   Hence the assumption that $$k_n |t_n|^3\|f_n\|^3=t^3\frac{\|f_n\|^3}{\sqrt{k_n}\sigma_n^3}\to 0$$
   implies
   \begin{equation*}
    \lambda_{f_n,t_n}^{k_n}\to e^{-\frac{t^2}{2}}.
   \end{equation*}
   This shows $$\lim_{n\to\infty}\int_\Omega e^{it\frac{S_n}{\sqrt{k_n}\sigma_n}}d\mu=e^{-\frac{1}{2}t^2},$$ finishing the proof of the theorem. 
 \end{proof}
%

We wish to point out that the proofs of the CLT presented in this section were detailed in the context of Gibbs-Markov maps but also hold in other, more general settings of mixing dynamical systems. 
The main technique is the spectral gap property of Ionescu-Tulcea and Marinescu ~\cite{IonescuTulceaMarinescu1950} which allows for the decomposition in equation \eqref{eq:trsfdecomp}. This property holds in all generality for maps which satisfy a Doeblin-Fortet inequality as in Lemma~\ref{lem:df}. Our Lemmas~\ref{lem:decomp}, \ref{lem:expansion} and \ref{lem:evexpansion} are instrumental in proving the CLT of Theorem~\ref{thm:CLTbd} and so this CLT holds in \emph{any} setting in which the above mentioned Lemmas are valid.

Note that identifying the context of a Banach space of functions along with a pair of norms satisfying our assumptions is a delicate but necessary task, without which the theorem lacks relevant examples, and we refrain from formulating our theorems in  an abstract albeit empty context. Other known examples of general settings to which Theorem~\ref{thm:CLTbd} applies, beyond the Gibbs-Markov systems presented in Section~\ref{sec:GibbsMarkov}, include maps of the interval endowed with the bounded variation norm, as well as Young towers endowed with the H\"older norm. For illustrative purposes we work out an example.

\begin{corollary}\label{cor:lesreg}
Let $(\Omega, \mu, T, \alpha)$ be a Gibbs-Markov map. Then every function
$$ f=\sum_{n=1}^\infty \gamma_n g_n$$
with $g_n\in L$, $\gamma_n\in\mathbb R$, $\int_\Omega g_n d\mu=0$, $\sup_{n\in\mathbb N} \|g_n\|_2<\infty$, $\sup_{n\in \mathbb N}|\gamma_n|\|g_n\|<\infty$ and $\sum_{k=n+1}^\infty |\gamma_k|\leqslant K n^{-3-\eta}$ (for some constants $K, \eta>0$) satisfies the central limit theorem in the form
$$ \frac 1{\sqrt{n}\sigma_n} \sum_{j=0}^{n-1} f\circ T^j \Rightarrow \mathcal N(0,1)$$
for some sequence $\sigma_n>0$, provided the asymptotic variances of $\sum_{k=1}^{n}\gamma_kg_k$ are bounded away from $0$ uniformly.
\end{corollary}

\begin{remark} If the asymptotic variances in the previous corollary converge to zero, then the ergodic sums normalized by $\sqrt n$ converge to $0$ stochastically.
\end{remark}

\begin{proof} Let $l_n\in \mathbb N$ satisfy
$$ \lim_{n\to\infty } n^{-1}l_n^6=0 \quad \mbox{and}\quad \lim_{n\to\infty} l_n^{-6-2\eta} n=0.$$

Define
$$ f_n=\sum_{j=1}^{l_n} \gamma_j g_j$$
and denote by $\sigma_n^2$ the asymptotic variance of $f_n$.
Note that $f_n\in L$ with $\|f_n\|\leqslant \sum_{k=1}^{l_n} |\gamma_k|\|g_k\|\leqslant C l_n$ for some constant $C>0$ and hence (since $\inf \sigma_n>0$)
$$ \frac{\|f_n\|^3}{\sqrt{n} \sigma_n^3} = O\left(\frac{l_n^3}{\sqrt{n}}\right)\to 0.$$

Take $k_n=n$ in Theorem \ref{thm:CLTbd} to deduce that
$$ \frac{f_n+f_n\circ T+...+f_n\circ T^{n-1}}{\sqrt{n}\sigma_{n}}$$
converges to the standard normal distribution.

Now by Chebychev's inequality with $M=\sup_{n\in \mathbb N} \|g_n\|_{2}$, for any $\epsilon>0$
\begin{eqnarray*} &&\mu\left(\left\{x\in \Omega: \left|\sum_{j=0}^{n-1} \sum_{k=l_n+1}^\infty \gamma_kg_k(T^j(x))\right|\geqslant \epsilon \sqrt{n}\sigma_n\right\}\right)\\
&\leqslant& \frac 1{\epsilon^2 n\sigma_n^2} \int_\Omega \left(\sum_{j=0}^{n-1} \sum_{k=l_n+1}^\infty  \gamma_k g_k\circ T^j \right)^2 d\mu\\
&\leqslant& \frac{1}{\epsilon^2 n\sigma_n^2} n^2M^2\left(\sum_{k=l_n+1}^\infty |\gamma_k|\right)^2\\
&\leqslant & O(n(l_n^{-3-\eta})^2),
\end{eqnarray*}
which converges to $0$. It follows that $\frac{1}{\sqrt{n}\sigma_n} \sum_{j=0}^{n-1} f_n\circ T^j$ and $\frac{1}{\sqrt{n}\sigma_n} \sum_{j=0}^{n-1} f\circ T^j$ have the same limiting distribution, whence the corollary.
\end{proof}

\begin{example}\label{ex:lesreg} Let $(\Omega,\mathcal B,\mu, T, \alpha)$ denote the continued fraction transformation with $\Omega=(0,1)$ and $\mu$ the Gauss measure. For every irrational $x\in(0,1)$, denote by $(x_n)_{n\in\mathbb N}$ its continued fraction expansion. Let $a_n:=\{x:x_0=n\}$ for every $n\in\mathbb N$, the partition $\alpha=\{a_n:n\in\mathbb N\}$. Let $\eta\in(0,\frac12)$, define
$$m_n:=\lfloor -\log_r n^{2}\rfloor, \quad \ell_n:=r^{-m_n},$$
$$\quad \gamma_n:= r^{(2+\eta)m_n}, \quad  g_n:=\ell_n{\bf 1}_{a_{\lfloor \ell_n\rfloor}}\circ T^{m_n}-\ell_n\mu(a_{\lfloor \ell_n\rfloor}).$$ Here we denote by $\lfloor\cdot \rfloor$ and $\lceil\cdot\rceil$ the usual floor and ceiling functions for real numbers.  Recall that $r\in (0,1)$ is the constant in \eqref{eq:distort}. In the current case, we can take $r=2/3$ (cf. \cite[Example 2]{AaronsonDenker2001}, noting $|T^2|'\geqslant 9/4$). It is easy to see that $g_n\in L$ and $\int_\Omega g_n d\mu=0$. 
$$\|g_n\|_\infty=\ell_n(1-\mu(a_{\lfloor \ell_n\rfloor})), \quad D_\Omega g_n=\ell_n r^{-m_n}=r^{-2m_n}\asymp n^4,$$
$$|\gamma_n|\|g_n\|=O(1)r^{\eta m_n}=O(1)n^{-2\eta},$$
$$\|g_n\|_2=\ell_n\sqrt{\mu(a_{\lfloor \ell_n\rfloor})(1-\mu(a_{\lfloor\ell_n\rfloor}))}=O(1)\ell_n /\sqrt{(\ell_n(\ell_n+1))},$$
\begin{multline*}\sum_{k=n}^\infty|\gamma_k|\leqslant\sum_{d=m_n}^\infty \sum_{\lceil r^{-d/2}\rceil\leqslant k\leqslant \lfloor r^{-(d+1)/2}\rfloor} r^{(2+\eta)d}\\ \leqslant \sum_{d=m_n}^\infty r^{(2+\eta)d}(r^{-(d+1)/2}-r^{-d/2})
\leqslant O(n^{-3-2\eta}).\end{multline*}
It follows that $f=\sum_{j=1}^\infty \gamma_j g_j$ satisfies the assumptions in the corollary, hence the central limit theorem holds:
$$ \frac 1{\sqrt{n}\sigma_n}\sum_{j=0}^{n-1} f\circ T^j \Rightarrow \mathcal N(0,1).$$
We remark that $f\in L^2(\mu)$ but $f\not\in L$. In fact, for large $n$, 
\begin{multline*}D_\alpha f\geqslant D_{\bigvee_{i=0}^{m_n-1}T^{-i}\alpha} f\geqslant r^{-m_n}\sum_{d\geqslant m_n} \sum_{\lceil r^{-d/2}\rceil\leqslant k< \lfloor r^{-(d+1)/2}\rfloor}r^{(2+\eta)d}\cdot r^{-d}\\
=r^{-m_n}O(r^{(\frac12+\eta)m_n})=O(r^{(-\frac12+\eta)m_n})=O(n^{1-2\eta}),\end{multline*}
hence $D_\alpha f$ is infinite. This calculation also indicates that for any $a\in\bigvee_{i=0}^{n-1}T^{-i}\alpha$, $D_a f$ is infinite.
\end{example}

 \section{A CLT for dynamical arrays after Lindeberg}
 \label{sec:lindeberg}
We prove a CLT for dynamical arrays, and later apply it to Birkhoff sums. The notion of dynamical array is also considered in \cite{DenkerSentiZhang2017}.

\begin{definition}
A dynamical array is a sequence $\{(F_{n,i},\tau_{n,i}): i=1,...,k_n\}_{n\in \mathbb N}$ consisting of a family of real valued functions $F_{n,i}$ defined on a dynamical system $(\Omega, T)$ and a family of initial times $\tau_{n,i}\in\mathbb N$, where $F_{n,i}$ is of form
$$ F_{n,i} =\sum_{j=1}^{l_{n,i}} f_{n,i,j}\circ T^{j-1},\qquad i=1,...,k_n$$ with the $f_{n,i,j}:\Omega\to \mathbb R$ and $l_{n,i}\in\mathbb N$
and where $\tau_{n,i}$ satisfies $\tau_{n,i-1}+ l_{n, i-1}\leqslant \tau_{n,i}$ for all $i=2,\ldots, k_n$. 
\end{definition}
Such a dynamical array brings about a sequence of sums $$F_{n,1}\circ T^{\tau_{n,1}}+F_{n,2}\circ T^{\tau_{n,2}}+\cdots+F_{n,k_n}\circ T^{\tau_{n,k_n}}\quad n\in\mathbb N.$$ Define $$m_n:=\inf_{2\leqslant i\leqslant k_n} \tau_{n,i}-\tau_{n,i-1}-l_{n,i-1}$$ as the minimal spacing (of the $n$-th row) of a dynamical array. 

We recall some notations. Let $(\Omega, \mu, T,\alpha)$ be a mixing Gibbs-Markov system. $\beta$ is a partition of $\Omega$ satisfying \eqref{eq:partbeta}, $\sigma(\beta)=\sigma(T\alpha)$. A real number $r$ inducing a metric on $\Omega$ is given in Definition \ref{def:gibbsmarkov}. The transfer operator $\mathcal L$ has the decomposition \eqref{eq:trsfdecomp} on $L=L_\beta^\infty$, i.e 
\begin{equation}\label{eq:trsfdecomp2}
\mathcal{L}f=\int_\Omega f d\mu+Nf\end{equation} for all $f\in L$. Let $\rho:=\frak{r}(N)\in (0,1).$
\begin{remark}
It is known (\cite[Corollaire 1]{Hennion1993}) that the essential spectrum of $\mathcal L$ is at most $r$ due to Lemma \ref{lem:df}, but the relation between $\rho$ and $r$ is unclear. \end{remark}

\begin{theorem}\label{thm:CLTdarray} Let $\{(F_{n,i},\tau_{n,i}): i=1,...,k_n\}_{n\in \mathbb N}$ be a dynamical array defined on a mixing Gibbs-Markov system $(\Omega, \mu, T,\alpha)$ with $F_{n,i} =\sum_{j=1}^{l_{n,i}} f_{n,i,j}\circ T^{j-1}$ and minimal spacing $m_n$. Suppose every $f_{n,i,j}\in L$ is centered, i.e. $\int_\Omega f_{n,i,j} d\mu=0$. Let $$\check{s}_n^2:=\Var(F_{n,1}\circ T^{\tau_{n,1}}+F_{n,2}\circ T^{\tau_{n,2}}+\cdots+ F_{n,k_n}\circ T^{\tau_{n,k_n}}).$$  Assume the following properties for this array.
\begin{enumerate}
\item For every $n\in\mathbb N$, $$\check{s}_n> 0.$$
\item 
\begin{equation}\label{eq:smallkdarray}
m_n>0 \quad { and }\quad \limsup_{n\to\infty} k_n^2 \rho^{m_n}<\infty.
\end{equation}
\item \begin{equation}\label{eq:condextradarray}
\limsup_{n\to\infty}\rho^{m_n}\sum_{1\leqslant i \leqslant k_n}r^{l_{n,i}}\frac{\|F_{n,i}\|}{\check{s}_n}<\infty.
\end{equation}
\item The Lindeberg condition holds, i.e. for every $\epsilon>0$ \begin{equation}\label{eq:Lindebergdarray}L_{n,\epsilon}:=\frac1{\check{s}_n^2} \sum_{i=1}^{k_n}\int_\Omega F_{n,i}^2\cdot {\bf 1}_{\{|F_{n,i}|\geqslant\epsilon\check{s}_n\}} d\mu\xrightarrow{n\to\infty} 0.\end{equation}
\end{enumerate}
Then, this array satisfies a CLT, i.e.
$$\frac{F_{n,1}\circ T^{\tau_{n,1}}+F_{n,2}\circ T^{\tau_{n,2}}+\cdots+ F_{n,k_n}\circ T^{\tau_{n,k_n}}}{\check{s}_n}\Rightarrow \mathcal N(0,1).$$
\end{theorem}
\begin{remark}
If $k_n\xrightarrow{n\to\infty}\infty$, condition \eqref{eq:smallkdarray} implies that $m_n\xrightarrow{n\to\infty}\infty$. 
\end{remark}
\begin{remark}
It will become clear in the proof that we can replace condition 3. by condition 3'.
$$\lim_{n\to\infty}\rho^{m_n}\sum_{1\leqslant i \leqslant k_n}r^{l_{n,i}}\frac{\|F_{n,i}\|}{\check{s}_n}\sup_{1\leqslant i \leqslant k_n}\frac{\|F_{n,i}\|_1}{\check{s}_n}=0$$
which will be handy to check in Section \ref{sec:stat}.
\end{remark}

\begin{lemma}
There exists a constant $C$ independent of $n$ such that with
$$\rho_n:=\rho^{m_n} \quad r_n:=r^{m_n},$$
for every $f\in L$, $n\in \mathbb N$ and $1\leqslant i<j\leqslant k_n$, if $\tau_{n,i+1}-\tau_{n,i}>l_{n,i}$ then
\begin{equation}\|N^{\tau_{n,j}-\tau_{n,i}}f\|\leqslant C\rho_n\|f\|_1+C\rho_n r_n^{j-i-1} r^{l_{n,i}+\cdots+l_{n,j-1}}  D_{\beta} f;\label{eq:estNdiff}\end{equation}
and if $\tau_{n,j}-\tau_{n,j-1}>l_{n,j-1}$ then
\begin{equation}\|N^{\tau_{n,j}}f\|\leqslant C\rho_n\|f\|_1+C\rho_n r^{\tau_{n,j-1}+l_{n,j-1}}  D_{\beta} f.\label{eq:estN}\end{equation}
\end{lemma}
\begin{proof} Because $\mathcal L=P+N$ with $PN=NP=0$, $$N^{\tau_{n,j}-\tau_{n,i}}f=N^{\tau_{n,i+1}-\tau_{n,i}-l_{n,i}}\mathcal L^{\tau_{n,j}-\tau_{n,i+1}+l_{n,i}}f.$$
Therefore, by \eqref{eq:df} (with $t=0$), we have \begin{align*}
\|N^{\tau_{n,j}-\tau_{n,i}}f\|&\leqslant\|N^{\tau_{n,i+1}-\tau_{n,i}-l_{n,i}}\|\|\mathcal L^{\tau_{n,j}-\tau_{n,i+1}+l_{n,i}}f\|\\
&\leqslant O(1) \rho^{\tau_{n,i+1}-\tau_{n,i}-l_{n,i}} (\|f\|_1+ r^{\tau_{n,j}-\tau_{n,i+1}+l_{n,i}} D_\beta f)\\
&\leqslant O(1) \rho^{m_n} (\|f\|_1+ r^{(j-i-1)m_n} r^{l_{n,i}+\cdots+l_{n, j-1}} D_\beta f).
\end{align*}
Similarly, for $2\leqslant j\leqslant k_n$, \eqref{eq:estN} follows from $$N^{\tau_{n,j}}f=N^{\tau_{n,j}-\tau_{n,j-1}-l_{n, j-1}}\mathcal L^{\tau_{n,j-1}+l_{n, j-1}}f.$$
\end{proof}

\begin{lemma} Under the same assumptions as in Theorem \ref{thm:CLTdarray},
\begin{enumerate}
\item  the array is asymptotically negligible,\begin{equation}\label{eq:supvar}\lim_{n\to\infty}\sup_{1\leqslant i\leqslant k_n}\int_{\Omega} \frac{F_{n,i}^2}{\check{s}_n^2}d\mu=0;
\end{equation}
\item  an asymptotic variance formula holds,\begin{equation}\label{eq:sumvar}
\lim_{n\to\infty}\sum_{i=1}^{k_n}\int_{\Omega} \frac{F_{n,i}^2}{\check{s}_n^2}d\mu=1.
\end{equation}
\end{enumerate}
\end{lemma} 
\begin{proof}
\begin{enumerate} 
\item This is implied by the Lindeberg condition \eqref{eq:Lindebergdarray},
for \begin{align}
\int_{\Omega} \frac{F_{n,i}^2}{\check{s}_n^2}d\mu&\leqslant \int_{\Omega} \frac{F_{n,i}^2}{\check{s}_n^2}\cdot{\bf 1}_{{\{|\frac{F_{n,i}}{\check{s}_n}|<\epsilon\}}}d\mu+ \sum_{i=1}^{k_n}\int_{\Omega} \frac{F_{n,i}^2}{\check{s}_n^2}\cdot{\bf 1}_{{\{|\frac{F_{n,i}}{\check{s}_n}|\geqslant\epsilon\}}}d\mu\notag\\
&\leqslant \epsilon^2+L_{n,\epsilon}.\label{eq:smallvar}\end{align}
\item Use the transfer operator $\mathcal L$ to expand the total variance \begin{align*}
\check{s}_n^2&=\int_{\Omega} (F_{n,1}\circ T^{\tau_{n,1}}+\cdots+F_{n,k_n}\circ T^{\tau_{n,k_n}})^2 d\mu\\
&=\sum_{i=1}^{k_n}\int_{\Omega} F_{n,i}^2d\mu+2\sum_{1\leqslant i<j\leqslant k_n} \int_{\Omega} F_{n,i}\cdot F_{n,j}\circ T_n^{\tau_{n,j}-\tau_{n,i}}d\mu\\
&=\sum_{i=1}^{k_n}\int_{\Omega} F_{n,i}^2d\mu+2\sum_{1\leqslant i<j\leqslant k_n} \int_{\Omega} \mathcal L^{\tau_{n,j}-\tau_{n,i}} F_{n,i}\cdot F_{n,j} d\mu\\
&\overset{\eqref{eq:trsfdecomp2}}{=}\sum_{i=1}^{k_n}\int_{\Omega} F_{n,i}^2d\mu+2\sum_{1\leqslant i<j\leqslant k_n} \int_{\Omega}  N^{\tau_{n,j}-\tau_{n,i}} F_{n,i}\cdot F_{n,j} d\mu.
\end{align*}
The last equality holds because $\int_\Omega F_{n,i}d \mu=0$. Estimate
\begin{align*}
&\quad \left|\frac{1}{\check{s}_n^2}\sum_{1\leqslant i<j\leqslant k_n} \int_\Omega N^{\tau_{n,j}-\tau_{n,i}} F_{n,i}\cdot F_{n,j}d\mu\right|\\
&\overset{\eqref{eq:estNdiff}}{\leqslant} \frac{1}{\check{s}_n^2}C\sum_{1\leqslant i<j\leqslant k_n} (\rho_n\|F_{n,i}\|_1+\rho_n r_n^{j-i-1} r^{l_{n,i}+\cdots+l_{n,j-1}} D_{\beta} F_{n,i})\cdot {\|F_{n,j}\|_1}\\
&\leqslant C \sup_{1\leqslant j\leqslant k_n}\left(\rho_n k_n^2\frac{\|F_{n,j}\|_1^2}{\check{s}_n^2}+\frac{\rho_n}{1-r_n}\sum_{1\leqslant i\leqslant k_n}r^{l_{n,i}}\frac{D_\beta F_{n,i}}{\check{s}_n}\frac{\|F_{n,j}\|_1}{\check{s}_n}\right)\\
&\leqslant C \sup_{1\leqslant j\leqslant k_n}\left(\rho_n k_n^2\frac{\|F_{n,j}\|_2^2}{\check{s}_n^2}+\frac{\rho_n}{1-r_n}\sum_{1\leqslant i\leqslant k_n}r^{l_{n,i}}\frac{\|F_{n,i}\|}{\check{s}_n}\frac{\|F_{n,j}\|_2}{\check{s}_n}\right)\\
&\overset{\eqref{eq:smallvar}}{\leqslant} C \rho_n k_n^2(\epsilon^2+ L_{n,\epsilon})+\frac{\rho_n}{1-r_n}\sum_{1\leqslant i\leqslant k_n}r^{l_{n,i}}\frac{\|F_{n,i}\|}{\check{s}_n}(\epsilon^2+ L_{n,\epsilon})^{1/2}.
\end{align*}
Now the assumptions \eqref{eq:smallkdarray}, \eqref{eq:condextradarray} and \eqref{eq:Lindebergdarray} imply that the $\limsup$ of the upper bound is bounded by $K\epsilon$ for some $K>0$, hence \eqref{eq:sumvar} follows.
\end{enumerate}
\end{proof}

\begin{proof}[Proof of \autoref{thm:CLTdarray}]
Extending our probability space if necessary, we may assume that there exists an array of random variables $\{X_{n,i}\}_{i=1}^{k_n}$ such that $X_{n,i}, i=1,\cdots, k_n,$ are independent normal random variables and 
\begin{equation}\label{eq:expvar}\mathbb E X_{n,i}=0 \text{ and } \Var X_{n,i}=\Var F_{n,i}.\end{equation} Without loss of generality we may as well assume that for each $n$, $\{X_{n,i}\}_{i=1}^{k_n}$ and $\{F_{n,i}\circ T^{\tau_{n,i}}\}_{i=1}^{k_n}$ are independent. Define two random variables \begin{align*}F_n&=\frac{F_{n,1}\circ T^{\tau_{n,1}}+F_{n,2}\circ T^{\tau_{n,2}}+\cdots+ F_{n,k_n}\circ T^{\tau_{n,k_n}}}{\check{s}_n},\\ X_n&= \frac{X_{n,1}+\cdots+X_{n,k_n}}{\check{s}_n}.\end{align*}
$X_n$ is a normal random variable for being a sum of independent normal random variables and converges weakly to $\mathcal N(0,1)$ because of \eqref{eq:sumvar}. Since $F_n$ has variance $1$, the set of distributions of $F_n$ is mass-preserving. To show that $F_n$ also converges weakly to $\mathcal N(0,1)$, it suffices to prove that for any $h$ in the separating class $C_c^\infty(\mathbb R)$,$$\mathbb E h(F_n)-\mathbb E h(X_n)\to 0.$$  Letting for $2\leqslant i\leqslant k_n-1$
$$U_{n,i}:=\frac{F_{n,1}\circ T^{\tau_{n,1}}+\cdots+F_{n,i-1}\circ T^{\tau_{n,i-1}}}{\check{s}_n}+\frac{X_{n,i+1}+\cdots+X_{n,k_n}}{\check{s}_n}
$$
and $$U_{n,1}=\frac{X_{n,2}+\cdots+X_{n,k_n}}{\check s_n} \qquad  U_{n,k_n}=\frac{F_{n,1}\circ T^{\tau_{n,1}}+\cdots+F_{n,k_n-1}\circ T^{\tau_{n,k_n-1}}}{\check s_n},$$
 we can write, noting that $F_n=U_{n,k_n}+\frac1{\check s_n}F_{n,k_n}\circ T^{\tau_{n,k_n}}$ and $X_n=U_{n,1}+\frac1{\check s_n}X_{n,1}$, $$h(F_n)-h(X_n)=\sum_{i=1}^{k_n}h\left(U_{n,i}+\frac{F_{n,i}\circ T^{\tau_{n,i}}}{\check{s}_n}\right)-h\left(U_{n,i}+\frac{X_{n,i}}{\check{s}_n}\right).$$
 Use Taylor expansion to deduce that
\begin{multline*}h(F_n)-h(X_n)=\sum_{i=1}^{k_n} h'(U_{n,i})\left(\frac{F_{n,i}\circ T^{\tau_{n,i}}}{\check{s}_n}-\frac{X_{n,i}}{\check{s}_n}\right)\\+h''\left(U_{n,i}+\theta_{n,i}\frac{F_{n,i}\circ T^{\tau_{n,i}}}{\check{s}_n}\right)\frac{F_{n,i}^2\circ T^{\tau_{n,i}}}{2\check{s}_n^2}-h''\left(U_{n,i}+\tilde{\theta}_{n,i}\frac{X_{n,i}}{\check{s}_n}\right)\frac{X^2_{n,i}}{2\check{s}_n^2},\end{multline*}
where $\theta_{n,i}, \tilde{\theta}_{n,i}:\Omega\to [0,1]$. Rewrite the right-hand side as
\begin{multline}\label{eq:hdiff}
\sum_{i=1}^{k_n} h'(U_{n,i})\left(\frac{F_{n,i}\circ T^{\tau_{n,i}}}{\check{s}_n}-\frac{X_{n,i}}{\check{s}_n}\right)+h''(U_{n,i})\left(\frac{F_{n,i}^2\circ T^{\tau_{n,i}}}{2\check{s}_n^2}-\frac{X^2_{n,i}}{2\check{s}_n^2}\right)\\+
\left\{h''\left(U_{n,i}+\theta_{n,i}\frac{F_{n,i}\circ T^{\tau_{n,i}}}{\check{s}_n}\right)\frac{F_{n,i}^2\circ T^{\tau_{n,i}}}{2\check{s}_n^2}-h''(U_{n,i})\frac{F_{n,i}^2\circ T^{\tau_{n,i}}}{2\check{s}_n^2}\right\}\\-\left\{h''\left(U_{n,i}+\tilde{\theta}_{n,i}\frac{X_{n,i}}{\check s_n}\right)\frac{X^2_{n,i}}{2\check{s}_n^2}-h''(U_{n,i})\frac{X^2_{n,i}}{2\check{s}_n^2}\right\}.
\end{multline}
We are about to show that the expectation of \eqref{eq:hdiff} vanishes asymptotically. Denote by $$\mathbb E_{n,i}(\cdot):=\mathbb E(\cdot|F_{n,1}\circ T^{\tau_{n,1}}, \ldots, F_{n,i}\circ T^{\tau_{n,i}})$$ the corresponding conditional expectation. To estimate the expectation of the first summand in \eqref{eq:hdiff}, we write
\begin{align}
&\quad \mathbb E \left(\sum_{i=1}^{k_n}h'(U_{n,i})\left(\frac{F_{n,i}\circ T^{\tau_{n,i}}}{\check{s}_n}-\frac{X_{n,i}}{\check{s}_n}\right)\right)\notag\\
&=\sum_{i=1}^{k_n}\mathbb E \left(\mathbb E_{n,i}\left(h'(U_{n,i})\right)\cdot\frac{F_{n,i}\circ T^{\tau_{n,i}}}{\check{s}_n}\right)-\mathbb E h'(U_{n,i})\mathbb E\frac{X_{n,i}}{\check{s}_n}\notag\\
&\overset{\eqref{eq:expvar}}{=}\frac{1}{\check{s}_n}\sum_{i=2}^{k_n}\int_\Omega \mathcal L^{\tau_{n,i}}\mathbb E_{n,i} h'(U_{n,i})\cdot F_{n,i}d\mu\notag\\
&\overset{\eqref{eq:trsfdecomp2}}{=}\frac{1}{\check{s}_n}\sum_{i=2}^{k_n}\left(\mathbb E h'(U_{n,i})\int_\Omega F_{n,i}d\mu+\int_\Omega N^{\tau_{n,i}} \mathbb E_{n,i} h'(U_{n,i}) \cdot F_{n,i} d\mu\right)\notag\\
&=\frac{1}{\check{s}_n}\sum_{i=2}^{k_n}\int_\Omega N^{\tau_{n,i}} \mathbb E_{n,i} h'(U_{n,i}) \cdot F_{n,i} d\mu, \label{eq:est1}
\end{align}
where in the first equality we use the independence between $U_{n,i}$ and $X_{n,i}$ and $$\mathbb E_{n,i}(h'(U_{n,i})\cdot F_{n,i}\circ T^{\tau_{n,i}})=\mathbb E_{n,i}(h'(U_{n,i}))\cdot F_{n,i}\circ T^{\tau_{n,i}},$$
in the second equality we also use that $U_{n,1}$ is independent with $F_{n,1}\circ T^{\tau_{n,1}}$
and the last equality is due to $\int_{\Omega} F_{n,i}d\mu=0$.
Observe the following inequalities.
\begin{enumerate}
\item By \eqref{eq:estN},\begin{align*}
&\quad \frac{1}{\check{s}_n}\sum_{i=2}^{k_n}\left|\int_\Omega N^{\tau_{n,i}} \mathbb E_{n,i} h'(U_{n,i}) \cdot F_{n,i} d\mu\right|\\
&\leqslant C\sum_{i=2}^{k_n} \left(\rho_n\|\mathbb E_{n,i} h'(U_{n,i})\|_1+\rho_n r^{\tau_{n,i-1}+l_{n,i-1}}D_{\beta} (\mathbb E_{n,i} h'(U_{n,i}))\right)\cdot\frac{\|F_{n,i}\|_1}{\check{s}_n}.
\end{align*}
\item
$$\|\mathbb E_{n,i} h'(U_{n,i})\|_1\les \mathbb E(\mathbb E_{n,i}|h'(U_{n,i})|)\les \|h'\|_\infty.$$
\item Recall that $U_{n,i}=\frac{1}{\check{s}_n}\left(\sum_{j=1}^{i-1}F_{n,j}\circ T^{\tau_{n,j}}+\sum_{j=i+1}^{k_n}X_{n,j}\right)$. Because $\{X_{n,j}\}_{j=1}^{k_n}$ and $\{F_{n,j}\circ T^{\tau_{n,j}}\}_{j=1}^{k_n}$ are independent,
$$
D_{\beta} (\mathbb E_{n,i} h'(U_{n,i}))
\les \frac{\|h''\|_{\infty}}{\check{s}_n}D_{\beta}(F_{n,1}\circ T^{\tau_{n,1}}+\cdots+F_{n,i-1}\circ T^{\tau_{n,i-1}}).$$
\item For any $f\in L$ and $m\in\mathbb N$,
\begin{align}D_\beta(f\circ T^m)&=\sup_{b\in\beta, x,y\in b}\frac{|f\circ T^m(x)-f\circ T^m(y)|}{r(x,y)}\notag\\
&=\sup_{b\in\beta, x,y\in b}\frac{|f\circ T^m(x)-f\circ T^m(y)|}{r(T^mx, T^my)}\frac{r(T^mx,T^my)}{r(x,y)}\notag\\
&\les D_\Omega f\cdot r^{-m}\notag\\
&\les \max\{\frac{2\|f\|_\infty}r, D_\beta(f) \}\cdot r^{-m}=O(1)\|f\|r^{-m}. \label{eq:cmpgrow}
\end{align}
\end{enumerate}
We use these inequalities to estimate \eqref{eq:est1},
\begin{align}
&\quad \frac{1}{\check{s}_n}\sum_{i=2}^{k_n}\left|\int_{\Omega_n} N^{\tau_{n,i}} \mathbb E_{n,i} h'(U_{n,i}) \cdot F_{n,i} d\mu\right|\notag \\
&\leqslant O(1) \sum_{i=2}^{k_n}\left(\rho_n\|h'\|_\infty+\rho_n r^{\tau_{n,i-1}+l_{n, i-1}}\frac{\|h''\|_{\infty}}{\check{s}_n} \sum_{j=1}^{i-1}\frac{1}{r^{\tau_{n,j}}}\|F_{n,j}\|\right)\cdot \frac{ \|F_{n,i}\|_1}{\check{s}_n}\notag \\
&\leqslant O(1)   \left(k_n\rho_n+ \rho_n\sum_{i=2}^{k_n}\sum_{j=1}^{i-1}r_n^{i-j-1}r^{l_{n,j}+\cdots + l_{n,i-1}}\frac{\| F_{n,j}\|}{\check{s}_n}\right)\cdot \sup_{1\leqslant i\leqslant k_n}\frac{\|F_{n,i}\|_2}{\check{s}_n}\notag \\
&\leqslant O(1)   \left(k_n\rho_n+ \frac{\rho_n}{1-r_n} \sum_{1\leqslant j\leqslant k_n}r^{l_{n,j}}\frac{\| F_{n,j}\|}{\check{s}_n}\right)\cdot \sup_{1\leqslant i\leqslant k_n}\frac{\|F_{n,i}\|_2}{\check{s}_n}.\label{eq:estcr}
\end{align}
The bound tends to $0$ as $n\to\infty$ because of \eqref{eq:supvar} and assumptions \eqref{eq:smallkdarray} and \eqref{eq:condextradarray}.

The expectation of the second summand in \eqref{eq:hdiff} is estimated in a similar way.
We rewrite \begin{align*}
&\quad\mathbb E \sum_{i=1}^{k_n}h''(U_{n,i})\left(F_{n,i}^2\circ T^{\tau_{n,i}} -X^2_{n,i}\right)\\
&=\sum_{i=2}^{k_n} \int_{\Omega} \mathbb E_{n,i}h''(U_{n,i})\cdot F^2_{n,i}\circ T^{\tau_{n,i}}d\mu-\mathbb E h''(U_{n,i}) \Var X_{n,i}\\
&\overset{\eqref{eq:trsfdecomp2}}{=}\sum_{i=2}^{k_n} \mathbb E h''(U_{n,i})\Var F_{n,i}+\int_\Omega N^{\tau_{n,i}}\mathbb E_{n,i}h''(U_{n,i})\cdot F_{n,i}^2d\mu-\mathbb E h''(U_{n,i}) \Var X_{n,i}\\
 &\overset{\eqref{eq:expvar}}{=}\sum_{i=2}^{k_n}\int_{\Omega} N^{\tau_{n,i}}\mathbb E_{n,i}h''(U_{n,i})\cdot F^2_{n,i}d\mu.
 \end{align*}
Then we can repeat the estimate for \eqref{eq:est1} to deduce an upper-bound similar to \eqref{eq:estcr}.
 
The expectation of the third summand in \eqref{eq:hdiff} is equal to
\begin{align*}
&\quad \sum_{i=1}^{k_n}\int_{\Omega} \left(h''\left(U_{n,i}+\theta_{n,i}\frac{F_{n,i}\circ T^{\tau_{n,i}}}{\check{s}_n}\right)-h''(U_{n,i})\right)\frac{F_{n,i}^2\circ T^{\tau_{n,i}}}{2\check{s}_n^2} d\mu\\
& =\sum_{i=1}^{k_n}\int_\Omega \left(h''\left(U_{n,i}+\theta_{n,i}\frac{F_{n,i}\circ T^{\tau_{n,i}}}{\check{s}_n}\right)-h''(U_{n,i})\right)\frac{F_{n,i}^2\circ T^{\tau_{n,i}}}{2\check{s}_n^2}\cdot \\
&\qquad   \left({\bf 1}_{\{|{F_{n,i}\circ T^{\tau_{n,i}}}|<\epsilon{\check{s}_n}\}}+{\bf 1}_{\{|{F_{n,i}\circ T^{\tau_{n,i}}}|\geqslant\epsilon{\check{s}_n}\}}\right)d\mu\\
&\leqslant  \epsilon \|h'''\|_\infty\sum_{i=1}^{k_n}\int_{\Omega} \frac{F_{n,i}^2}{\check{s}_n^2}d\mu + \|h''\|_\infty L_{n,\epsilon}
\end{align*}
for any $\epsilon>0$. This expectation converges to $0$ in view of  \eqref{eq:sumvar} and \eqref{eq:Lindebergdarray}. The expectation of the last summand in \eqref{eq:hdiff} is controlled in the same way as the third summand.
 \end{proof}
 
Applying this theorem to Birkhoff sums, we obtain the following result.
\begin{corollary}\label{thm:CLTone} Given a Gibbs-Markov system $(\Omega, \mu, T, \alpha)$ and a sequence of centered functions $\{f_n\}$ in $L$. Let $s_n^2:=\Var(f_n+\cdots+f_n\circ T^{n-1})$. Assume that there are sequences of integers $l_n>m_n>0$ with the following properties.
\begin{enumerate}
\item $$\limsup_{n\to\infty}k_n^2\rho^{m_n}<\infty,$$ where $k_n:=[\frac{n}{l_n+m_n}]$.
\item \begin{equation}\label{eq:ergsmallnorm}
\limsup_{n\to\infty}k_n\rho^{m_n}\frac{\|f_n\|}{s_n}<\infty.
\end{equation} 
\item For every $1\leqslant i\leqslant l_n$\begin{equation}\label{eq:ergsmallvar}
\frac{1}{s_n^2}\int_\Omega (f_n+\cdots+f_n\circ T^{i-1})^2d\mu\to 0.
\end{equation}
\item \begin{equation}\label{eq:erggapvar}
\frac{k_n}{s_n^2}\int_\Omega (f_n+\cdots+f_n\circ T^{m_n-1})^2d\mu\to 0.
\end{equation}
\item \begin{equation}\label{eq:ui}
\frac{k_n}{s_n^2}(f_n+\cdots+f_n\circ T^{l_n-1})^2 \text{ is uniformly integrable.}
\end{equation}
\end{enumerate}
Then $\frac{f_n+\cdots+f_n\circ T^{n-1}}{s_n}\Rightarrow\mathcal N(0,1).$ 
\end{corollary}
\begin{proof}
Let $F_{n,i}:=f_n+\cdots+f_n\circ T^{l_n-1}, g_{n,i}:=f_n+\cdots+f_n\circ T^{m_n-1}$ and $\tau_{n,i}:=(i-1)(l_n+m_n)$ for $1\leqslant i\leqslant k_n$, then $$f_n+\cdots+f_n\circ T^{k_n(l_n+m_n)-1}=\sum_{i=1}^{k_n} F_{n,i}\circ T^{\tau_{n,i}}+g_{n,i}\circ T^{\tau_{n,i}+l_n}.$$ 
To complete the ergodic sum,  let $F_{n,k_n+1}:=f_n+\cdots +f_n\circ T^{\min\{l_n, n-k_n(l_n+m_n)\}-1}$, $\tau_{n,k_n+1}:=k_n(l_n+m_n)$ and $g_{n,k_n+1}:=f_n+\cdots+f_n\circ T^{n-k_n(l_n+m_n)-l_n-1}$ if necessary.
The following two properties ensure that the dynamical array $\{(F_{n,i},\tau_{n,i}): i=1,\ldots, k_n+1\}$ has the same distributional limit as the ergodic sum.
\begin{enumerate}
\item \begin{equation}\label{eq:simvar}\frac{\check{s}_n}{s_n}\to 1,\end{equation} where $\check{s}_n^2=\Var \sum_{i=1}^{k_n+1} F_{n,i}\circ T^{\tau_{n,i}}$.
\item \begin{equation}\label{eq:gapdistr}\frac{\sum_{i=1}^{k_n+1}g_{n, i} \circ T^{\tau_{n,i}+l_n}}{\check{s}_n}\Rightarrow 0.\end{equation}
\end{enumerate}
In fact, to see \eqref{eq:simvar} first note that
$$s_n^2=\check{s}_n^2+\Var \sum_{i=1}^{k_n+1} g_{n,i}\circ T^{\tau_{n,i}+l_n}+2 \int_\Omega \sum_{i=1}^{k_n+1} F_{n,i}\circ T^{\tau_{n,i}}\cdot\sum_{i=1}^{k_n+1} g_{n,i}\circ T^{\tau_{n,i}+l_n}d\mu.$$
With conditions \eqref{eq:ergsmallnorm} and \eqref{eq:erggapvar}, arguments involving the transfer operator similar to those used in proving \eqref{eq:sumvar} indicate that $$\lim_{n\to\infty} \frac{\Var \sum_{i=1}^{k_n+1} g_{n,i}\circ T^{\tau_{n,i}+l_n}}{s_n^2}=\lim_{n\to\infty} \sum_{i=1}^{k_n+1}\frac{\int_\Omega g_{n,i}^2d\mu}{s_n^2},$$
which is $0$ by \eqref{eq:ergsmallvar} and \eqref{eq:erggapvar}. As we can separate
\begin{multline*}\int_\Omega F_{n,i}\circ T^{\tau_{n,i}}\cdot\sum_{j=1}^{k_n+1} g_{n,j}\circ T^{\tau_{n,j}+l_n}d\mu=\sum_{j\ges i+1}\int_\Omega F_{n,i}\cdot g_{n,j}\circ T^{\tau_{n,j}+l_n-\tau_{n,i}}d\mu
\\+\sum_{j\les i-2}\int_\Omega g_{n,j}\cdot F_{n,i}\circ T^{\tau_{n,i}-\tau_{n,j}-l_n}d\mu
+\int_\Omega F_{n,i}\cdot g_{n,i}\circ T^{l_n}d\mu+\int_\Omega g_{n,i-1}\cdot F_{n,i}\circ T^{m_n}d\mu,
\end{multline*}
$$\int F_{n,i}\cdot g_{n,i}\circ T^{l_n}d\mu=\int (f_n+\cdots+f_n\circ T^{l_n-m_n-1})\cdot g_{n,i}\circ T^{l_n}d\mu+\int g_{n,i}\cdot g_{n,i}\circ T^{m_n}d\mu$$
and similarly for $\int_\Omega g_{n,i-1}\cdot F_{n,i}\circ T^{m_n}d\mu$, the techniques of transfer operator can be used again to show that$$\frac{1}{s_n^2}\int_\Omega \sum_{i=1}^{k_n+1}F_{n,i}\circ T^{\tau_{n,i}}\cdot\sum_{j=1}^{k_n+1} g_{n,j}\circ T^{\tau_{n,j}+l_n}d\mu\to 0.$$
Thus \eqref{eq:simvar} holds. The previous arguments also imply that \eqref{eq:gapdistr} is just a consequence of \eqref{eq:erggapvar}. Hence we only need to verify the conditions in Theorem \ref{thm:CLTdarray} for the dynamical array $\{(F_{n,i},\tau_{n,i}): i=1,\ldots, k_n+1\}$.  \eqref{eq:condextradarray} is taken care of by \eqref{eq:ergsmallnorm} since
$$\limsup_{n\to\infty}\rho^{m_n}\sum_{1\les i \les k_n}r^{l_n}\frac{\|F_{n,i}\|}{\check{s}_n}\overset{\eqref{eq:cmpgrow}}{\les} O(1)\limsup_{n\to\infty}\rho^{m_n} k_n\frac{\|f_n\|}{\check{s}_n}.$$
Note that $$\sum_{i=1}^{k_n}\int_\Omega \frac{F_{n,i}^2}{\check{s}_n^2} {\bf 1}_{\{|F_{n,i}|\ges\epsilon\check{s}_n\}} d\mu=k_n\int_\Omega  \frac{F_{n,1}^2}{\check s_n^2} {\bf 1}_{\{|F_{n,1}|\ges\epsilon\check{s}_n\}} d\mu,$$
the Lindeberg condition \eqref{eq:Lindebergdarray} follows from \eqref{eq:ui}, \eqref{eq:ergsmallvar} and \eqref{eq:simvar}.
\end{proof}

\begin{remark}Theorem \ref{thm:CLTdarray} also can be generalized with the same assumptions to more general dynamical systems. It in fact holds for any system for which the transfer operator satisfies the Doeblin-Fortet inequality \eqref{eq:df} and for which the composition operator satisfies the inequality \eqref{eq:cmpgrow} (or in the case of Lipschitz norm, $r(T^mx,T^my)\les\frac{r(x,y)}{r^m}$). Note that the inequalities \eqref{eq:df} and \eqref{eq:cmpgrow} are bounded at the rates of $r$ and $r^{-1}$ respectively.
\end{remark}

\section{Applications to the large sample theory in statistics}\label{sec:stat}
The CLT under the Lindeberg condition has many applications, in particular in nonparametric statistics. The book \cite{Ferguson1996} provides a glimpse on these applications, though it is not a complete list. Here we restrict to one particular case, the famous Behrens-Fisher problem. In what follows, consider the setup for the two sample problem in a Gibbs-Markov dynamical system $(\Omega, \mu, T, \alpha )$. 
\begin{definition}\label{def:CLTlin}Denote by $\tilde{L}$ the set of all measurable functions $f:\Omega\to \mathbb R$ for which there exists a sequence of functions $\{f_n:n\geqslant 1\}$ in $L$ such that $\|f-f_n\|_2\to 0$.
\end{definition}

Consider two functions $\phi,\psi:\Omega\to\mathbb R$ in the class $\tilde{L}$, which determine two stationary sequences $X_n=\phi\circ T^n$ and $Y_n=\psi\circ T^n$.  For simplicity we assume that the distributions $\mu_\phi$ of $\phi$ and $\mu_\psi$ of  $\psi$ have no atoms. Based on observations $X_1,...,X_m$ and $Y_1,...,Y_n$, the Behrens-Fisher problem is  to determine whether the distributions  of $\phi$ and  $\psi$ are different or not in a statistical sense.  We shall deal with this problem when the distributions differ in their means, that is $\int \phi d\mu\ne \int \psi d\mu$.
  
   The classical solution for this problem (to be the most powerful test) is the $t$-test which only works  exactly under normal distribution, independence and equal variances.  In all other cases some type of approximation is needed. In particular, when the distributions of $\phi$ and $\psi$ are completely  unknown, the two sample Wilcoxon rank sum test is widely used.
   Consider $m,n\in \mathbb N$ and observations $X_1,...,X_m$ and $Y_1,...,Y_n$. Define $R_i$ to be the rank of $X_i$ among all $n+m$ observations $X_1,...,X_m,Y_1,...,Y_n$. Then
   $$ W_{m,n}= \sum_{i=1}^m R_i$$
   is  the two sample  Wilcoxon rank sum test. In order to solve the problem in a nonparametric setup one needs to determine the asymptotic distribution of $W_{m,n}$.
     
 \begin{align*}
 & W_{m,n}=  \sum_{i=1}^m \sum_{k=1}^n {\bf 1}_{\{Y_k\les X_i\}}+\sum_{i=1}^m\sum_{k=1}^m {\bf 1}_{\{X_k\les X_i\}}\\
 & = \sum_{i=1}^m \sum_{k=1}^n {\bf 1}_{\{Y_k\les X_i\}} + \frac {m(m-1)}2\\
 & =\left(\sum_{i=1}^m\sum_{k=1}^n \left({\bf 1}_{\{Y_k\les X_i\}}- \int{\bf 1}_{\{Y_k\les t\}} \mu_\phi(dt) -\int{\bf 1}_{\{t\les X_i\}} \mu_\psi(dt) + \iint{\bf 1}_{\{s\les t\}}\mu_\phi(dt)\mu_\psi(ds)\right)\right)\\
 & \quad + m\left(\sum_{k=1}^n\int {\bf 1}_{\{Y_k\les t\}} \mu_\phi(dt) -  n\iint {\bf 1}_{\{s\les t\}}\mu_\phi(dt)\mu_\psi(ds)\right)\\
 & \quad +n\left(\sum_{i=1}^m\int {\bf 1}_{\{t\les X_i\}} \mu_\psi(dt) -  m\iint {\bf 1}_{\{s\les t\}}\mu_\phi(dt)\mu_\psi(ds)\right)\\
 & \quad +\left(mn\iint{\bf 1}_{\{s\les t\}}\mu_\phi(dt)\mu_\psi(ds) + \frac {m(m-1)} 2\right)\\
 & =: A+mB_n+nC_m+D.
 \end{align*}
 We first give conditions under which the second moment of $A$, normalized by $m^{3}$  converges to zero as $m\to \infty$ and $ n/m\to \lambda\in (0,1)$. This can be seen directly or by applying  \cite{DenkerGordin2014} when  $(x,y)\mapsto {\bf 1}_{\{\psi(y)\les \phi(x)\}}$ approximately belongs to the projective tensor product $L_{2,\pi}(\mu^2)$ over $L_2(\mu^2)$ and therefore the variance of the approximation $\tilde A$ to $A$ increases like $\sqrt{nm}\|\tilde A\|_{L_{2,\pi}(\mu^2)}$.  We refer to \cite{DenkerGordin2014} for the definitions and properties of projective tensor products. Alternatively, assuming that the distributions of $\psi$ and $\phi$ are absolutely continuous with respect to Lebesgue measure and have a bounded density, one could use \cite[Theorem 1 or Lemma 3]{DenkerKeller1986} to show that the variance of $A$ is of smaller order. As an example, we prove
 \begin{proposition} Assume that the distributions $\mu_\phi$ and $\mu_\psi$ satisfy
 $$ \mu_\phi(I) \les K \eta^r\quad \mbox{ and}\quad \mu_\psi(I)\les K\eta^r\qquad \forall\eta>0, \forall\text{interval $I$ of length $\eta$}$$
 for some $K>0$ and $r\in(\frac 45, 1]$ and assume that $\|\phi\|_\infty$ and $\|\psi\|_\infty$ are finite. Then, as $n/m\to\lambda\in (0,1)$, $A$ has a representation $A=A_1+A_2$ so that $$\Var A_1=o(m^3)\quad \text{ and } \quad \mathbb E|A_2|=o(m^{3/2}).$$ Therefore, normalized by $m^{3/2}$, $A$ does not contribute to the distributional limit of $W_{m,n}$.
 \end{proposition}
 
 \begin{proof} 
Let $m\in \mathbb N$ and choose $q$ which depends on $m$ and is chosen below.
 Let $M=\max\{\|\phi\|_\infty, \|\psi\|_\infty\}$ and $h(x,y)= {\bf 1}_{\{-M\les y\les x\les M\}}$. Divide the interval $[-M,M]$ into $q$ subintervals $J_1,...,J_{q}$ of equal length $2Mq^{-1}$ and let
$$  I_j=\{(x,y): x\in J_j, -M\les y\les \min J_j\},\qquad  I= \bigcup_{j=1}^{q} I_j. $$
 Then $\mu_\phi\times \mu_\psi(\{(x,y):-M\les y\les x\les M\}\setminus I)\les K(2M)^r q^{-r}$ and the projective norm of $(u,v) \mapsto \tilde h_q(u,v)={\bf 1}_I(\phi(u),\psi(v))$ is bounded by (cf. \cite[Lemma 1]{DenkerGordin2014})
$$\|\tilde h_q\|_{L_{4,\pi}(\mu^2)}\les \sum_{j=1}^{q} \|{\bf 1}_{I_j}(\phi,\psi)\|_{L_4(\mu^2)}\les q \left[K(2M)^rq^{-r}\right]^{1/4}.$$
 Write
 $$ \hat h_q(u,v)=  \tilde h_q(u,v) -\int \tilde h_q(w,v) \mu(dw) -\int \tilde h_q(u,w)\mu(dw) +\iint \tilde h_q(w,w')\mu(dw)\mu(dw') $$
 and 
 $$ A_1=\sum_{k=1}^n\sum_{i=1}^m \hat h_q(T^i(u),T^k(v)).$$
 Then $\|\hat h_q\|_{L_{4,\pi}(\mu^2)}= O(\|\tilde h_q\|_{L_{4,\pi}(\mu^2)})$ and applying  Lemma 4 in \cite{DenkerGordin2014} with $d=m=2$ and $p=4$ (one can verify the assumption in this lemma for $\hat h_q$) it follows that there is a constant  $C$ (independent of $q$ and $(n,m)$) such that
 $$ \left\| A_1\right\|_{L_2(\mu)} \les C \sqrt{nm}\|\hat h_q\|_{L_{4,\pi}(\mu^2)} =O\left(m q^{1-r/4}\right).$$
 Moreover, we get
 $$ \iint \left|\sum_{k=1}^n\sum_{i=1}^m  h(\phi(T^i(u)),\psi(T^k(v)))-\tilde h_q (T^i(u),T^k(v)) \right| \mu(du)\mu(dv)\les K (2M)^r q^{-r}nm.$$
 Similar estimates hold for the other summand $A_2=A-A_1$.
 
 Since $1\geqslant r>\frac 45$, we have that $0<2-\frac r 2< 2r$ hence can pick
 $$\frac 1{2r}<\tau < \frac 1{2-\frac r2}$$
 and $q= m^\tau$ to obtain
 $$ m^{-3} \mathbb E(A_1^2)= O(m^{-3}m^2q^{2-r/2})=O(m^{-1+2\tau-\frac {r\tau}2})= o(1)$$
 and
 $$  m^{-3/2}\mathbb E|A-A_1|=O(m^{-3/2}m^2 q^{-r}) = O(m^{\frac 12 -r\tau})= o(1). $$
 \end{proof}

 Since $\phi,\psi\in \tilde L$ they are approximated in $L^2$ by functions $\phi_m,\psi_n\in L$. Set $F_\phi$ and $F_\psi$ for the respective distribution functions of $\phi$ and $\psi$. Denote
 \begin{eqnarray*}
  \tilde B_n &:=& \sum_{k=1}^n (1-F_\phi(\psi_n\circ T^k)) - n\int (1- F_\phi(s)) \mu_{\psi_n}(ds)\\
 \tilde C_m&:=&\sum_{k=1}^m  F_\psi(\phi_m\circ T^k) -m\int F_\psi(s) \mu_{\phi_m}(ds)
 \end{eqnarray*}
 and
$$ \sigma_m^2:= m^2 \Var(\tilde B_{n(m)}) + n(m)^2  \Var(\tilde C_m) + 2 n(m)m \mbox{\rm Cov}(\tilde B_{n(m)}, \tilde C_m).$$ 
 We are not developing more details and extensions of the forgoing discussions, instead we assume that
 \begin{equation}\label{eq:ass1}
 \mbox{\rm Var}(A)=o(\sigma_m^2)
 \end{equation}
 \begin{equation}\label{eq:ass2}
 \|\phi-\phi_m\|_1 =o( n^{-1}m^{-1}\sigma_m)\qquad \|\psi-\psi_n\|_1=o(n^{-1}m^{-1}\sigma_m)
 \end{equation}
 \begin{equation}\label{eq:ass3}
  D_\alpha\phi_m=o(m^2) \qquad D_\alpha\psi_n=o(n^2)
 \end{equation}
 and that $F_\psi$ and $F_\phi$ are Lipschitz continuous. 
Under these simplifying assumptions the following argument becomes short and shows the pattern of the proof under  more general assumptions.
 
 \begin{proposition} Under the assumptions \eqref{eq:ass1} and \eqref{eq:ass2} and \eqref{eq:ass3} and if $n=n(m)$ so that 
 for some $\lambda\in (0,1)$, $\lambda\les n/m\les\lambda^{-1}$
and that 
 \begin{equation}\label{eq:varstatlind}
 \liminf \sigma_m m^{-3/2}>0,
 \end{equation} then as $m\to \infty$
 $$ \frac 1{\sigma_m} \left(W_{m,n(m)}- D \right)\Rightarrow \mathcal N(0,1).$$
 \end{proposition}
Note that in case  the distributions  of $\phi$ and $\psi$ are equal, then 
 $$D= \frac{mn}{2}+ \frac {m(m-1)}{2}= \frac m2(n+m-1).$$
 This shows that the two sample Wilcoxon rank sum test checks whether the distributions of $\phi$ and $\psi$ differ by a location alternative.

 \begin{proof}  We first show that $ \frac 1{\sigma_m} (W_{m,n(m)}-D)$ and 
 $ \frac 1{\sigma_m}(m\tilde B_{n(m)} + n(m)\tilde C_m)$ have the same limiting distribution. This follows from the assumption (\ref{eq:ass1}) and  (using (\ref{eq:ass2}))
 $$ m\|B_{n(m)}-\tilde B_{n(m)}\| _1 \les  2n(m)mD_{F_\phi} \|\psi-\psi_{n(m)}\|_1 =o(\sigma_m)$$ 
where $D_{F_\phi}$ denotes the corresponding Lipschitz constant and from a similar inequality for $n\|C_m-\tilde C_m\|$.
%
Hence the assertion of the proposition follows  if 
 $$  V_m=\frac{1}{\sigma_m} (m\tilde B_{n(m)} +    n(m) \tilde C_m)$$
 converges weakly to the standard  normal distribution.

We apply Theorem \ref{thm:CLTdarray}. Let w.l.o.g. $n\les m$, $\theta_{\phi_m}=\iint{\bf 1}_{\{\psi\les \phi_m\}} d\mu_\psi d\mu_{\phi_m}$, $\theta_{\psi_n}=\iint{\bf 1}_{\{\phi\les \psi_n\}} d\mu_{\psi_n} d\mu_\phi$, $n= k_n(p+q)+q_n$ and $m= k_m(p+q) + q_m$ where $0\les q_n, q_m< p+q$. Denote
 \begin{eqnarray*}
 && F_{n,i}= \sum_{l=0}^{p-1} \left[n (F_\psi(\phi_m\circ T^l)-\theta_{\phi_m})-{m}(F_\phi(\psi_n\circ T^l)-\theta_{\psi_n})\right]\quad 1\les i\les k_n,\\
 && F_{n,i}= \sum_{l=0}^{p-1}\left[n (F_\psi(\phi_m\circ T^l)-\theta_{\phi_m}) \right] \qquad k_n+1\les i\les k_m,\\
 && \tau_{n,i}= (i-1)(p+q)\quad i=1,...,k_m,\quad \check s_{n}^2=\Var(\sum_{i=1}^{k_m}{F_{n,i}\circ T^{\tau_n,i}}).
 \end{eqnarray*}
 
We check next that conditions 1.--4. in Theorem \ref{thm:CLTdarray} hold with an appropriate choice of $p$ and $q$. Let $r$ and $\rho$ be the constants related to the decomposition \eqref{eq:trsfdecomp2} which are given by the system. Choosing $p=O(m^{\frac12-\delta})$ for some $0<\delta<\frac 12$ and $q =\lfloor\frac{2\log m}{-\log \rho}\rfloor$ it follows that
 $$ k_m^2\rho^{q}=O(m^{-1+2\delta}),$$
 hence 2. holds.  Since $\|F_{n,i}\|\les r^{-p} (nD_\alpha(F_\psi\circ \phi_m) +mD_\alpha(F_\phi\circ \psi_n)) + O(mp)$ for $i=1,...,k_n$ and similarly for $i=k_n+1,...,k_m$ we calculate:
 \begin{align*} \rho^q k_m r^p \|F_{n,i}\|\|F_{n,i}\|_1 \sigma_m^{-2}&= O(m^{-\frac 32+\delta} (m(D\phi_m+D\psi_n)+ r^pmp)mp\sigma_m^{-2})\\
 &=O((D_\alpha\phi_m + D_\alpha\psi_n)m^{-2})\overset{\eqref{eq:ass3}}{\to}0\end{align*}
and $$ \int F_{n,i}^2 {\bf 1}_{\{|F_{n,i}| \ges \epsilon\check s_n\}} d\mu \les {\epsilon^{-2} \check s_n^{-2}}\int F_{n,i}^4 d\mu =  O(m^2p^2) \check{s}_n^{-2}  \int F_{n,i}^2 d\mu$$
   and hence $$\sigma_m^{-2}\sum_{i=1}^{k_m}\int F_{n,i}^2 {\bf 1}_{\{|F_{n,i}| \ges \epsilon\check s_n\}} d\mu=O(m^{-2\delta})\check s_n^{-2}\sum_{i=1}^{k_m}\int F_{n,i}^2d\mu.$$ It is straightforward to show using the calculus developed in this article (see \eqref{eq:sumvar} and \eqref{eq:simvar} and   observing \eqref{eq:varstatlind} and that $|F_\psi\circ \phi_m|$ and $|F_\phi\circ\psi_n|$ are bounded) that
$$ \check s_n/\sigma_m = 1+ O(\sqrt{D\phi_m + D\psi_n}/ m)\overset{\eqref{eq:ass3}}{\to} 1$$
and
$$ \check s_n^{-2} \sum_{i=1}^{k_m} \int F_{n,i}^2 d\mu \to 1 .$$ Therefore conditions 1., 3'. and 4. hold.
   
   It is proved now that
   $$ \frac 1{\sigma_m} \sum_{i=1}^{k_m} F_{n,i}$$
   converges weakly to the standard normal distribution. We finally remark that $\frac 1{\sigma_m} \sum_{i=1}^{k_m} F_{n,i}$ is stochastically equivalent to $V_m$, since the variance of the difference is bounded by
 $$ \sigma_m^{-2} k_m q^2 m^2= O(m^{-\frac 12+\delta} (\log m)^2 )=o(1)$$
were one uses the same estimates as for the comparison of $\check s_n$ and $\sigma_m$. This finishes the proof. 
 \end{proof}
 
{\noindent{\bf Acknowledgment:}
{The research was supported  by  {\it n\'umero 158/2012 de Pesquisador Visitante Especial de CAPES}. XZ was also supported by PNPD/CAPES.}
}

\bibliographystyle{plain}


\begin{thebibliography}{10}

\bibitem{AaronsonDenker1999}
Jon Aaronson and Manfred Denker.
\newblock The {P}oincar\'e series of {$\mathbf C\setminus\mathbf Z$}.
\newblock {\em Ergodic Theory Dynam. Systems}, 19(1):1--20, 1999.

\bibitem{AaronsonDenker2001}
Jon Aaronson and Manfred Denker.
\newblock Local limit theorems for partial sums of stationary sequences
  generated by {G}ibbs-{M}arkov maps.
\newblock {\em Stoch. Dyn.}, 1(2):193--237, 2001.

\bibitem{AaronsonDenkerUrbanski1993}
Jon Aaronson, Manfred Denker, and Mariusz Urba{\'n}ski.
\newblock Ergodic theory for {M}arkov fibred systems and parabolic rational
  maps.
\newblock {\em Trans. Amer. Math. Soc.}, 337(2):495--548, 1993.

\bibitem{Chazottes2015}
Jean-Ren{\'e} Chazottes.
\newblock Fluctuations of observables in dynamical systems: from limit theorems
  to concentration inequalities.
\newblock In {\em Nonlinear dynamics new directions}, volume~11 of {\em
  Nonlinear Syst. Complex.}, pages 47--85. Springer, Cham, 2015.

\bibitem{CohenConze2017}
Guy Cohen and Jean-Pierre Conze.
\newblock {CLT} for random walks of commuting endomorphisms on compact abelian
  groups.
\newblock {\em Journal of Theoretical Probability}, 30(1):143--195, 2017.

\bibitem{ConzeRaugi2007}
Jean-Pierre Conze and Albert Raugi.
\newblock Limit theorems for sequential expanding dynamical systems on
  {$[0,1]$}.
\newblock In {\em Ergodic theory and related fields}, volume 430 of {\em
  Contemp. Math.}, pages 89--121. Amer. Math. Soc., Providence, RI, 2007.

\bibitem{Denker1982}
Manfred Denker.
\newblock Statistical decision procedures and ergodic theory.
\newblock In {\em Ergodic theory and related topics ({V}itte, 1981)}, volume~12
  of {\em Math. Res.}, pages 35--47. Akademie-Verlag, Berlin, 1982.

\bibitem{Denker1989}
Manfred Denker.
\newblock The central limit theorem for dynamical systems.
\newblock In {\em Dynamical systems and ergodic theory ({W}arsaw, 1986)},
  volume~23 of {\em Banach Center Publ.}, pages 33--62. PWN, Warsaw, 1989.

\bibitem{DenkerGordin2014}
Manfred Denker and Mikhail Gordin.
\newblock Limit theorems for von {M}ises statistics of a measure preserving
  transformation.
\newblock {\em Probab. Theory Related Fields}, 160(1-2):1--45, 2014.

\bibitem{DenkerKeller1986}
Manfred Denker and Gerhard Keller.
\newblock Rigorous statistical procedures for data from dynamical systems.
\newblock {\em J. Statist. Phys.}, 44(1-2):67--93, 1986.

\bibitem{DenkerPrzytyckiUrbanski1996}
Manfred Denker, Feliks Przytycki, and Mariusz Urba{\'n}ski.
\newblock On the transfer operator for rational functions on the {R}iemann
  sphere.
\newblock {\em Ergodic Theory Dynam. Systems}, 16(2):255--266, 1996.

\bibitem{DenkerSentiZhang2017}
Manfred Denker, Samuel Senti, and Xuan Zhang.
\newblock Fluctuations of ergodic sums over periodic orbits under
  specification.
\newblock preprint, 2017.

\bibitem{DunfordSchwartz1958}
Nelson Dunford and Jacob~T. Schwartz.
\newblock {\em Linear {O}perators. {I}. {G}eneral {T}heory}.
\newblock With the assistance of W. G. Bade and R. G. Bartle. Pure and Applied
  Mathematics, Vol. 7. Interscience Publishers, Inc., New York; Interscience
  Publishers, Ltd., London, 1958.

\bibitem{Ferguson1996}
Thomas~S. Ferguson.
\newblock {\em A course in large sample theory}.
\newblock Texts in Statistical Science Series. Chapman \& Hall, London, 1996.

\bibitem{Gordin1969}
M.~I. Gordin.
\newblock The central limit theorem for stationary processes.
\newblock {\em Dokl. Akad. Nauk SSSR}, 188:739--741, 1969.

\bibitem{Goueezel2010}
S{\'e}bastien Gou{\"e}zel.
\newblock Characterization of weak convergence of {B}irkhoff sums for
  {G}ibbs-{M}arkov maps.
\newblock {\em Israel J. Math.}, 180:1--41, 2010.

\bibitem{HaydnNicolVaientiEtAl2013}
Nicolai Haydn, Matthew Nicol, Sandro Vaienti, and Licheng Zhang.
\newblock Central limit theorems for the shrinking target problem.
\newblock {\em J. Stat. Phys.}, 153(5):864--887, 2013.

\bibitem{Hennion1993}
Hubert Hennion.
\newblock Sur un th\'eor\`eme spectral et son application aux noyaux
  lipchitziens.
\newblock {\em Proc. Amer. Math. Soc.}, 118(2):627--634, 1993.

\bibitem{IbragimovLinnik1971}
I.~A. Ibragimov and Yu.~V. Linnik.
\newblock {\em Independent and stationary sequences of random variables}.
\newblock Wolters-Noordhoff Publishing, Groningen, 1971.
\newblock With a supplementary chapter by I. A. Ibragimov and V. V. Petrov,
  Translation from the Russian edited by J. F. C. Kingman.

\bibitem{IonescuTulceaMarinescu1950}
C.~T. Ionescu~Tulcea and G.~Marinescu.
\newblock Th\'eorie ergodique pour des classes d'op\'erations non
  compl\`etement continues.
\newblock {\em Ann. of Math. (2)}, 52:140--147, 1950.

\bibitem{Lindeberg1922}
J.~W. Lindeberg.
\newblock Eine neue {H}erleitung des {E}xponentialgesetzes in der
  {W}ahrscheinlichkeitsrechnung.
\newblock {\em Math. Z.}, 15(1):211--225, 1922.

\bibitem{Rousseau-Egele1983}
J.~Rousseau-Egele.
\newblock Un th\'eor\`eme de la limite locale pour une classe de
  transformations dilatantes et monotones par morceaux.
\newblock {\em Ann. Probab.}, 11(3):772--788, 1983.

\bibitem{Thomine2014}
Damien Thomine.
\newblock A generalized central limit theorem in infinite ergodic theory.
\newblock {\em Probab. Theory Related Fields}, 158(3-4):597--636, 2014.

\bibitem{Young1998}
Lai-Sang Young.
\newblock Statistical properties of dynamical systems with some hyperbolicity.
\newblock {\em Ann. of Math. (2)}, 147(3):585--650, 1998.

\bibitem{Zhang2015}
Xuan Zhang.
\newblock {\em Studies on the weak convergence of partial sums in Gibbs-Markov
  dynamical systems}.
\newblock PhD thesis, The Pennsylvania State University, 2015.

\end{thebibliography}
\end{document}